\documentclass[reqno,11pt]{amsart}
\usepackage{setspace,tikz,xcolor,mathrsfs,listings,multicol,amssymb}
\usepackage{rotating}
\usepackage[vcentermath]{youngtab}
\usepackage[margin=1in,includefoot,footskip=30pt]{geometry}
\usepackage{enumerate}
\usepackage{booktabs}
\usepackage[centertableaux]{ytableau}
\usepackage[all,cmtip]{xy}
\usetikzlibrary{arrows,matrix}
\tikzset{tab/.style={matrix of math nodes,column sep=-.35, row sep=-.35,text height=7pt,text width=7pt,align=center,inner sep=2,font=\footnotesize}}

\usepackage[colorlinks=true, pdfstartview=FitV, linkcolor=blue, citecolor=blue, urlcolor=blue]{hyperref}

\newcommand{\arxiv}[1]{\href{https://arxiv.org/abs/#1}{\texttt{arXiv:#1}}}

\newcommand{\inner}[2]{\left\langle #1, #2 \right\rangle}
\newcommand{\normord}[1]{: \mathrel{#1} :}  

\newcommand{\abs}[1]{\left\lvert #1 \right\rvert}

\newcommand{\bra}[1]{\langle #1 \rvert}
\newcommand{\ket}[1]{\lvert #1 \rangle}
\newcommand{\braket}[2]{\langle #1 | #2 \rangle}
\newcommand{\ds}{/\!\!/}  

\newcommand{\field}{\mathbf{k}}

\newcommand{\G}{G}
\newcommand{\dG}{g}
\newcommand{\wG}{J}
\newcommand{\dwG}{j}

\DeclareMathOperator{\id}{id} 

\newcommand{\mcA}{\mathcal{A}}
\newcommand{\mcD}{\mathcal{D}}
\newcommand{\mcE}{\mathcal{E}}
\newcommand{\mcF}{\mathcal{F}}
\newcommand{\mcH}{\mathcal{H}}
\newcommand{\mcI}{\mathcal{I}}
\newcommand{\mcP}{\mathcal{P}}

\newcommand{\vv}{\mathbf{v}}
\newcommand{\xx}{\mathbf{x}}
\newcommand{\yy}{\mathbf{y}}

\newcommand{\bal}{\boldsymbol{\alpha}}
\newcommand{\bbe}{\boldsymbol{\beta}}

\newcommand{\ZZ}{\mathbb{Z}}
\newcommand{\QQ}{\mathbb{Q}}

\newcommand{\CC}{\mathbb{C}}

\definecolor{darkred}{rgb}{0.7,0,0} 
\newcommand{\defn}[1]{{\color{darkred}\emph{#1}}} 

\definecolor{UQgold}{RGB}{196, 158, 54} 
\definecolor{UQpurple}{RGB}{73, 7, 94} 
\definecolor{UMNgold}{RGB}{255,200,46} 
\definecolor{UMNmaroon}{RGB}{106,0,50} 
\definecolor{OCUenji}{RGB}{153,0,51} 
\definecolor{OCUsapphire}{RGB}{0,51,102} 
\definecolor{TUblue}{RGB}{0,77,255} 

\usepackage{listings}
\lstdefinelanguage{Sage}[]{Python}
{morekeywords={False,sage,True},sensitive=true}
\definecolor{dblackcolor}{rgb}{0.0,0.0,0.0}
\definecolor{dbluecolor}{rgb}{0.01,0.02,0.7}
\definecolor{dgreencolor}{rgb}{0.2,0.4,0.0}
\definecolor{dgraycolor}{rgb}{0.30,0.3,0.30}

\theoremstyle{plain}
\newtheorem{thm}{Theorem}[section]
\newtheorem{lemma}[thm]{Lemma}

\newtheorem{prop}[thm]{Proposition}
\newtheorem{cor}[thm]{Corollary}
\theoremstyle{definition}
\newtheorem{dfn}[thm]{Definition}
\newtheorem{ex}[thm]{Example}
\newtheorem{remark}[thm]{Remark}

\numberwithin{equation}{section}

\usepackage[colorinlistoftodos]{todonotes}

\begin{document}
\title[Canonical free-fermions]{Free-fermions and canonical Grothendieck polynomials}

\author[S.~Iwao]{Shinsuke Iwao}
\address[S.~Iwao]{Faculty of Business and Commerce, Keio University, Hiyosi 4--1--1, Kohoku-ku, Yokohama-si, Kanagawa 223-8521, Japan}
\email{iwao-s@keio.jp}

\author[K.~Motegi]{Kohei Motegi}
\address[K.~Motegi]{Faculty of Marine Technology, Tokyo University of Marine Science and Technology, Etchujima 2--1--6, Koto-Ku, Tokyo, 135-8533, Japan}
\email{kmoteg0@kaiyodai.ac.jp}
\urladdr{https://sites.google.com/site/motegikohei/home}

\author[T.~Scrimshaw]{Travis Scrimshaw}
\address[T.~Scrimshaw]{Faculty of Science, Hokkaido University, 5 Ch\=ome Kita 8 J\=onishi, Kita Ward, Sapporo, Hokkaid\=o 060-0808}
\email{tcscrims@gmail.com}
\urladdr{https://tscrim.github.io/}


\thanks{
S.I.~was partially supported by Grant-in-Aid for Scientific Research (C) 19K03605.
K.M.~was partially supported by Grant-in-Aid for Scientific Research (C) 21K03176, 20K03793.
T.S.~was partially supported by Grant-in-Aid for JSPS Fellows 21F51028.
}

\begin{abstract}
We give a presentation of refined (dual) canonical Grothendieck polynomials and their skew versions using free-fermions.
Using this, we derive a number of identities, including the skew Cauchy identities, branching rules, expansion formulas, and integral formulas.
\end{abstract}

\maketitle
\tableofcontents

\section{Introduction}
\label{sec:introduction}

The (symmetric) Grothendieck functions $G_{\lambda}(\xx; \beta)$, where $\lambda$ is a partition inside a $k \times (n-k)$ rectangle, are symmetric functions used to study the K-theory of the Grassmannian, the set of $k$-dimensional subspaces in $\CC^n$, arising from the work of Lascoux and Sch\"utzenberger~\cite{LS82,LS83} with the $\beta$ parameter introduced by Fomin and Kirillov~\cite{FK94}.
(Strictly speaking, they live in a completion of the ring of symmetric functions, but this distinction is insignificant for our purposes.)
Many formulas are known for Grothendieck functions, such as a ratio of alternates~\cite[Eq.~(1.8)]{LS83} (see also~\cite[Eq.~(2.3)]{Lenart00}) and a sum over combinatorial objects~\cite{Buch02}.
An important property is that $\G_{\lambda}(\xx; \beta)$ is equal to $s_{\lambda}(\xx)$, the classical Schur function, plus higher degree terms, which was first given explicitly by Lenart~\cite{Lenart00}.
As such, they form a basis for (a completion of) $\Lambda$ and we can form the dual basis $\{\dG_{\lambda}(\xx; \beta)\}_{\lambda}$ under the (continuously extended) Hall inner product
$\inner{s_{\lambda}}{s_{\mu}} = \delta_{\lambda\mu}$.
We can also apply the involution $\omega$, which acts by $\omega s_{\lambda} = s_{\lambda'}$ with $\lambda'$ the conjugate shape, to each basis, which produces new bases that are called the weak versions.
These other bases were first studied by Lam and Pylyavskyy~\cite{LP08}, where they were given combinatorial descriptions.

Galashin, Grinberg, and Liu~\cite{GGL16} refined the parameter $\beta$ into a sequence of parameters $\bbe = (\beta_1, \beta_2, \ldots)$ that record additional combinatorial information for the dual Grothendieck functions.
The refined dual Grothendieck functions $\dG_{\lambda/\mu}(\xx; \bbe)$ have been used to describe properties of a last-passage percolation model in probability theory~\cite{MS20}, which is a refinement of~\cite{Yel20}.
The refined version of Grothendieck functions $\G_{\lambda/\mu}(\xx; \bbe)$ were introduced in~\cite{CP21} with applications to Brill--Noether varieties from algebraic geometry.
In a separate direction, Yeliussizov introduced the canonical Grothendieck functions $\G_{\lambda}(\xx; \alpha, \beta)$ in~\cite{Yel17} being inspired by canonical bases for Hecke algebras (more commonly known as Kazhdan--Lusztig bases) using $\omega$ as the defining involution.
These specialize to the usual Grothendiecks at $\alpha = 0$ and the weak Grothendiecks at $\beta = 0$, and similarly for the dual versions.
These generalizations were combined into the refined canonical Grothendieck functions $\G_{\lambda}(\xx; \bal, \bbe)$ and their dual version $\dG_{\lambda}(\xx; \bal, \bbe)$ by Hwang \textit{et al.}~\cite{HJKSS21} (they go even further and define flagged versions).
They proved Jacobi--Trudi formulas that specialize to other known formulas~\cite{AY20,Kim20II,Kim20,MS20,Yel17} and the refined version of Yeliussizov's symmetry
\begin{equation}
\label{eq:symmetries}
\omega \G_{\lambda/\mu}(\xx, \bal, \bbe) = \G_{\lambda'/\mu'}(\xx; \bbe, \bal),
\qquad\qquad
\omega \dG_{\lambda/\mu}(\xx, \bal, \bbe) = \dG_{\lambda'/\mu'}(\xx; \bbe, \bal).
\end{equation}
For simplicity, from this point onward, we will drop the adjective ``refined'' in the nomenclature.

Our main result (Theorem~\ref{thm:jacobi_trudi}, Corollary~\ref{cor:free_fermion_grothendiecks}) in this paper is a description of the canonial Grothendieck functions and their duals using the boson-fermion correspondence (see, \textit{e.g.},~\cite{AZ13,DKJM83,kac90,KKR13}).
In more detail, a certain infinite wedge space known as (fermionic) Fock space $\mcF$ has an action of the infinite dimensional Heisenberg algebra (a bosonic action) that we use to construct half-vertex operators $e^{H(\xx)}$ in terms of a Hamiltonian using the powersum symmetric functions.
However, Fock space $\mcF$ is also a Clifford algebra representation (which is where the fermionic name comes from), so we can express our half-vertex operators as an action of the Clifford algebra.
This correspondence allows us to write determinant formulas for the pairing of certain vectors and half-vertex operators by Wick's theorem.
We define a new basis $\ket{\lambda}^{[\bal,\bbe]}$ of a certain subspace $\mcF^0$ of Fock space and its dual basis ${}^{[\bal,\bbe]} \bra{\lambda}$ (under the natural pairing) in Theorem~\ref{thm:dual_basis} by generalizing the first author's previous work~\cite{Iwao19,Iwao20,Iwao23}.
By the natural Clifford algebra involution, we also have another orthonormal basis $\ket{\lambda}_{[\bal,\bbe]}$ and its dual ${}_{[\bal,\bbe]} \bra{\lambda}$.
Then taking the Jacobi--Trudi formulas of~\cite{HJKSS21} as our \emph{definition} of the skew (dual) canonical Grothendieck functions, we show the matrix elements
\[
\G_{\lambda\ds\mu}(\xx; \bal, \bbe) = {}^{[\bal,\bbe]} \bra{\mu} e^{H(\xx)} \ket{\lambda}^{[\bal,\bbe]},
\]
where $G_{\lambda\ds\mu}(\xx; \bal, \bbe)$ is the refined version of the corresponding functions of~\cite{Buch02,Yel17}, and
\[
\dG_{\lambda/\mu}(\xx; \bal, \bbe) = {}_{[\bal,\bbe]} \bra{\mu} e^{H(\xx)} \ket{\lambda}_{[\bal,\bbe]}
\]
(Theorem~\ref{thm:jacobi_trudi}, Corollary~\ref{cor:free_fermion_grothendiecks}).
We also introduce modified vectors ${}^{[[\bal,\bbe]]}\bra{\mu}$ so that
\[
\G_{\lambda/\mu}(\xx; \bal, \bbe) = {}^{[[\bal,\bbe]]} \bra{\mu} e^{H(\xx)} \ket{\lambda}^{[\bal,\bbe]}.
\]

Using our free-fermion description, we obtain a number of additional results.
We give a new simple proof the symmetries~\eqref{eq:symmetries} in Theorem~\ref{thm:canonical_conjugate} and the Schur expansion formulas for $\dG_{\lambda/\mu}(\xx; \bal, \bbe)$ (Theorem~\ref{thm:matrix_elements}) from~\cite[Thm.~8.7,~8.8]{HJKSS21}.
We also have expansion formulas for $\G_{\lambda/\mu}(\xx; \bal, \bbe)$, which are similar to but slightly different from those in~\cite[Thm.~8.2, 8.3]{HJKSS21}, and $\G_{\lambda\ds\mu}(\xx; \bal, \bbe)$.
A number of other results are generalized, such as branching rules~\cite[Prop.~8.7,8.8]{Yel17} (Proposition~\ref{prop:branching_rules}), the skew Cauchy identities~\cite[Thm.~5.1,Cor.~6.3]{Yel19}
 (Theorem~\ref{thm:skew_cauchy}),
the skew Pieri-type identities~\cite[Thm.~7.10]{Yel19} (Theorem \ref{thm: pieri-formulas}),
the formula in~\cite[Rem.~2.8]{HIMN17} (Proposition~\ref{prop:determinant_formula_for_Gs}), and integral formulas in~\cite[Prop.~4.28]{MS20} (Theorems~\ref{thm:integral_G} and~\ref{thm:integral_Gconj}).
We show determinant formulas for the expansions of $\G_{\lambda}(\xx; \bal, \bbe)$ into $\G_{\mu}(\xx; 0, \bbe)$ and similarly for the dual and give combinatorial descriptions of the coefficients, which allows us to answer~\cite[Prob.~12.2]{Yel17} in the negative (Section~\ref{sec:one_parameter_decomposition}).
Our last result is a free-fermionic presentation for a special case of the flagged canonical Grothendieck functions in Proposition~\ref{prop:flagged_fermions} with a flagged version of Proposition~\ref{prop:determinant_formula_for_Gs}, which is a canonical version of~\cite[Thm.~4]{Matsumura19}.

Let us briefly digress to discuss vertex models as there is a well-known relationship with free-fermions; see, \textit{e.g.},~\cite{Hardt21} and references therein.
There is a vertex model known for Grothendieck functions~\cite{MS13,MS14,WZJ16,ZinnJustin09}.
However, this lattice model is not at the free-fermion point (a condition on the weights), so we cannot go between the two descriptions.
This extends to the recent work of Gunna and Zinn-Justin~\cite{GZJ23}, where they gave a vertex model for canonical Grothendieck functions, and we cannot establish a direct relationship with our results.

This is also the first in a series of papers where we study the relationship between Grothendieck functions and stochastic processes.
In our next paper~\cite{IMSprob}, we will use our free-fermion presentation to study the four variants of the totally asymmetric simple exclusion process (TASEP) studied by Dieker and Warren in~\cite{DW08}.
Indeed, we can already see that appropriate specializations of the Jacobi--Trudi formulas are precisely, up to a simple overall factor, the transition kernels in~\cite{DW08}, which was first noticed in~\cite{MS20}.
In particular, we will extend the noncommutative operators given in~\cite{Iwao19,Iwao20} to the refined settings and show these encode the dynamics of the particle motions.
In~\cite{IMScomb}, we will show that we can recover the combinatorial description of canonical Grothendieck functions and their duals by using branching rules to reduce the computation to a single variable, which requires more technical analysis.

This paper is organized as follows.
In Section~\ref{sec:background}, we give some background on supersymmetric functions and the boson-fermion correspondence.
In Section~\ref{sec:canonical_fermions}, we describe new vectors in fermionic Fock space and prove a number of properties.
In Section~\ref{sec:results}, we prove our main results and identities.
In Section~\ref{sec:flagged}, we give a free-fermionic presentation of a special case of the flagged canonical Grothendieck functions.

\subsection*{Acknowledgements}

The authors thank Darij Grinberg and Jang Soo Kim for useful conversions.
The authors thank Ole Warnaar on the behalf of Alain Lascoux for posthumous comments on an earlier draft of this paper.
The authors thank the referees for their useful comments.

This work benefited from computations using {\sc SageMath}~\cite{sage,combinat}.
This work was partly supported by Osaka City University Advanced Mathematical Institute (MEXT Joint Usage/Research Center on Mathematics and Theoretical Physics JPMXP0619217849).
This work was supported by the Research Institute for Mathematical Sciences, an International Joint Usage/Research Center located in Kyoto University.

\section{Background}
\label{sec:background}

Let $\lambda = (\lambda_1, \lambda_2, \dotsc, \lambda_{\ell})$ be a \defn{partition}, a weakly decreasing finite sequence of positive integers.
We denote the set of all partitions by $\mcP$.
We draw the Young diagrams of our partitions using English convention.
We will often extend partitions with additional entries at the end being $0$, and let $\ell(\lambda)$ denote the largest index $\ell$ such that $\lambda_{\ell} > 0$.
Let $\lambda'$ denote the conjugate partition.
We often write our partitions as words.
A \defn{hook} is a partition $\lambda$ of the form $a1^{m} = (a, 1, \dotsc, 1)$ with $1$ appears $m$ times, where the \defn{arm} is $a-1$ and the \defn{leg} is $m$.

Let $\xx = (x_1, x_2, \ldots)$ denote a countably infinite sequence of indeterminates.
We will often set all but finitely many of the indeterminates $\xx$ to $0$, which we denote as $\xx_n := (x_1, \dotsc, x_n, 0, 0, \ldots)$.
We make similar definitions for another sequence of indeterminates $\yy = (y_1, y_2, \ldots)$.
We also require infinite sequences of parameters $\bal = (\alpha_1, \alpha_2, \ldots)$, and $\bbe = (\beta_1, \beta_2, \ldots)$, which we often treat as indeterminates.

\subsection{Supersymmetric functions}

We set some additional standard notation from symmetric function theory.
Let
\[
e_m(\xx) = \sum_{i_1 < \cdots < i_m} x_{i_1} \dotsm x_{i_m},
\qquad
h_m(\xx) = \sum_{i_1 \leq \cdots \leq i_m} x_{i_1} \dotsm x_{i_m},
\qquad p_m(\xx) = \sum_{i=1}^{\infty} x_i^m,
\]
denote the elementary, homogeneous, and power sum, respectively.
For $\lambda\in \mathcal{P}$, we set
$e_{\lambda}(\xx) = e_{\lambda_1} \cdots e_{\lambda_{\ell}}$,
$h_{\lambda}(\xx) = h_{\lambda_1} \cdots h_{\lambda_{\ell}}$, and 
$p_{\lambda}(\xx) = p_{\lambda_1} \cdots p_{\lambda_{\ell}}$.
Let $\Lambda_{\QQ}$ be the algebra of symmetric functions over $\QQ$. It is known that
\[
\Lambda_{\QQ}=\QQ[h_1(\xx),h_2(\xx),\dots]=\QQ[e_1(\xx),e_2(\xx),\dots]=\QQ[p_1(\xx),p_2(\xx),\dots].
\]
We can define the polynomials $E_{\lambda}(p_1, p_2, \ldots)$, $H_{\lambda}(p_1, p_2, \ldots)$, $S_{\lambda}(p_1, p_2, \ldots)$ with coefficients in $\QQ$ by the equations
\[ 
E_{\mu}(p_1(\xx), p_2(\xx), \ldots) = e_{\mu}(\xx),
\qquad
H_{\mu}(p_1(\xx), p_2(\xx), \ldots) = h_{\mu}(\xx),
\qquad
S_{\mu}(p_1(\xx), p_2(\xx), \ldots) = s_{\mu}(\xx).
\] 

Now we recall some particular \defn{supersymmetric functions}; we refer the reader to~\cite[Ch.~I]{MacdonaldBook} for more details.
We define the supersymmetric elementary, homogeneous, power sum, and Schur functions as
\begin{align*}
e_m(\xx/\yy)  = \sum_{k=0}^m (-1)^{m-k} e_k(\xx) h_{m-k}(\yy),\qquad
h_m(\xx/\yy)  = \sum_{k=0}^m (-1)^{m-k} h_k(\xx) e_{m-k}(\yy),
\\ p_m(\xx/\yy)  = p_m(\xx) - p_m(\yy),\qquad
 s_{\lambda}(\xx/\yy)  = \sum_{\mu} (-1)^{\abs{\lambda} - \abs{\mu}} s_{\mu}(\xx) s_{\lambda' / \mu'}(\yy).
\end{align*}
When $\yy = 0$, that is we have set all of the $\yy$ indeterminates to $0$, we have $f(\xx / \yy) = f(\xx)$ for any supersymmetric function $f$.

From~\cite[Sec.~I.5.Ex.~23]{MacdonaldBook}, we have $s_{\lambda/\mu}(\xx/\yy) = (-1)^{\abs{\lambda/\mu}} s_{\lambda'/\mu'}(\yy/\xx) = s_{\lambda'/\mu'}({-\yy}/{-\xx})$, and we can consider the supersymmetric Schur functions as the sum over bitableaux of shape $\lambda$, which are fillings of $\lambda$ by elements in the totally ordered set
\[
\{1 < 2 < 3 < \cdots < 1' < 2' < 3' < \cdots\}
\]
such that rows and columns weakly increase and no primed (resp.~unprimed) letter is repeated in the same row (resp.~column).
Therefore, we have
\[
s_{\lambda}(\xx/\yy) = \sum_T x_1^{T_1} x_2^{T_2} \cdots (-y_1)^{T_{1'}} (-y_2)^{T_{2'}} \cdots,
\]
and have a natural definition of the supersymmetric skew Schur function $s_{\lambda/\mu}(\xx/\yy)$ analogous to the skew Schur function $s_{\lambda}(\xx)$.

The supersymmetric functions can also be described in terms of plethystic substitution.
While we will not give a detailed account, we will briefly review the relevant descriptions for understanding the results in~\cite{HJKSS21} and refer the reader to~\cite{LR11} and~\cite[Ch.~I]{MacdonaldBook} for a more detailed description.
Let $X = x_1 + x_2 + \cdots$ and $Y = y_1 + y_2 + \cdots$.
For a symmetric function $f$, we define $f[X] = f(x_1, x_2, \ldots)$, and if $Z = z_1 + z_2 + \cdots + z_n$, then we have $f[Z] = f(z_1, z_2, \dotsc, z_n, 0, 0, \ldots)$.
We also can define
\begin{align*}
h_m[X - Y] = h_m(\xx/\yy),
\qquad\qquad
e_m[X - Y] = e_m(\xx/\yy),
\qquad\qquad
p_m[X - Y] = p_m(\xx/\yy).
\end{align*}
As a consequence, we have that $h_m[-Y] = (-1)^m e_m(\yy)$ and $e_m[-Y] = (-1)^m h_m(\yy)$.
Furthermore, we have
\begin{align*}
h_m\bigl((\xx \sqcup \xx')/(\yy \sqcup \yy')\bigr) & = h_m[X + X' - Y - Y'] = \sum_{a+b=m} h_a(\xx/\yy) h_b(\xx'/\yy'),
\\
e_m\bigl((\xx \sqcup \xx')/(\yy \sqcup \yy')\bigr) & = e_m[X + X' - Y - Y'] = \sum_{a+b=m} e_a(\xx/\yy) e_b(\xx'/\yy'),
\end{align*}
which are well-known identities (see, \textit{e.g.},~\cite[Prop.~2.1]{HJKSS21}).
Next, we recall the notation given in~\cite[Def.~2.4]{HJKSS21}:
\[
h_m[X \ominus Y] := \sum_{a-b=m} h_a[X] h_b[Y],
\qquad\qquad
e_m[X \ominus Y] := \sum_{a-b=m} e_a[X] e_b[Y].
\]
We note that these can have infinite nonzero terms and be nonzero even when $m$ is negative.

In order to avoid confusion with the plethystic negative and negating the variables, we will not use plethystic notation, and instead write $h_m(\xx \ds \yy) := h_m[X \ominus Y]$ and $e_m(\xx \ds \yy) := e_m[X \ominus Y]$.

\subsection{Free-fermions and Wick's theorem}

We describe free-fermions 
and Wick's theorem here.

Let $\field$ be a field of characteristic $0$.
The ($\field$-)algebra of \defn{free-fermions} $\mcA$ is the associative unital $\field$-algebra generated by $\{\psi_n, \psi_n^* \mid n \in \ZZ\}$ subject to the anti-commuting relations
\[
\psi_m \psi_n + \psi_n \psi_m = \psi_m^* \psi_n^* + \psi_n^* \psi_m^* = 0,
\qquad\qquad
\psi_m \psi_n^* + \psi_n^* \psi_m = \delta_{m,n}.
\]
This is the Clifford algebra that often appears in mathematical physics for the vector space with a basis indexed by $\ZZ \sqcup \ZZ$.
There exists an anti-algebra involution on $\mcA$ defined by $\psi_n \leftrightarrow \psi_n^\ast$ satisfying $(xy)^\ast=y^\ast x^\ast$ for any $x,y\in \mathcal{A}$.

We define \defn{(fermionic) Fock space} $\mcF$ as the subspace of $\bigwedge^{\infty} V$, where $V = \bigoplus_{i \in \ZZ} \field v_i$, with the basis
\[
\{ v_{i_1} \wedge v_{i_2} \wedge \cdots \mid i_1 > i_2 > \cdots, \; i_k = -k+m \text{ for sufficiently large $k$ and some } m \in \ZZ \}.
\]
We make $\mcF$ into a left $\mcA$-representation generated by the \defn{vacuum vector}
\[
\ket{0} = v_{-1} \wedge v_{-2} \wedge v_{-3} \wedge \cdots
\]
with the action of $\mcA$ given by
\begin{align*}
\psi_n(v_{i_1} \wedge v_{i_2} \wedge \cdots) & = v_n \wedge v_{i_1} \wedge v_{i_2} \wedge \cdots,
\\ \psi_n^*(v_{i_1} \wedge v_{i_2} \wedge \cdots) & = \begin{cases} (-1)^{k-1} v_{i_1} \wedge \cdots \wedge v_{i_{k-1}} \wedge v_{i_{k+1}} \wedge \cdots & \text{if there exists $k$ such that $i_k = n$}, \\ 0 & \text{otherwise,} \end{cases}
\end{align*}
and extended linearly.
We can see that $\mcF$ satisfies the relations
\[
\psi_n \ket{0} = \psi_m^* \ket{0} = 0, \qquad\qquad n < 0, \quad m \geq 0,
\]
and the basis can be described by the vectors
\[
\psi_{n_1} \psi_{n_2} \cdots \psi_{n_r} \psi_{m_1}^* \psi_{m_2}^* \cdots \psi_{m_s}^* \ket{0},
\qquad
(r,s \geq 0, n_1 > \cdots n_r \geq 0 > m_s > \cdots > m_1).
\]
We refer the reader to~\cite[Sec.~5]{KRR13} and~\cite[Sec.~4]{MJD00} for more information.

In the language of Clifford algebras, Fock space is a spinor representation of $\mcA$.
The basis elements of Fock space corresponds to a single (semi-infinite) wedge product of basis vectors $v_{i_1} \wedge v_{i_2} \wedge \cdots$.
We can consider such basis elements as being a state in a system of particles on a line with particles at positions $i_1$, $i_2$, and so on.
Thus, the vacuum vector corresponds to the state where all of the particles are to the left of the position $-1/2$, which is called the step initial condition in the asymmetric simple exclusion process (ASEP) or the Dirac sea in the mathematical physics literature.

We can similarly define the dual Fock space $\mcF^*$ as the right $\mcA$-representation generated by the (dual) vacuum vector
\[
\bra{0} = \cdots \wedge v_2 \wedge v_1 \wedge v_0
\]
with the action of $\mcA$ given by
\begin{align*}
\psi_n(\cdots \wedge v_{i_2} \wedge v_{i_1}) & = \begin{cases} (-1)^{k-1}  \cdots \wedge v_{i_{k+1}} \wedge v_{i_{k-1}} \wedge \cdots \wedge v_{i_1} \! & \text{if there exists $k$ such that $i_k = n - 1$}, \\ 0 & \text{otherwise,} \end{cases}
\\ \psi_n^*(\cdots \wedge v_{i_2} \wedge v_{i_1}) & = \cdots \wedge v_{i_2} \wedge v_{i_1} \wedge v_{n-1}.
\end{align*}
We can see that the roles of $\psi_j$ and $\psi_j^*$ have been reversed compared to $\mcF$ (albeit with the indices trivially shifted by one for their action on the wedge space).
Therefore, dual Fock space $\mcF^*$ satisfies the relations
\[
\bra{0} \psi_m = \bra{0} \psi_n^* = 0, \qquad\qquad n < 0, \quad m \geq 0,
\]
has a basis given by
\[
\bra{0} \psi_{m_s} \cdots \psi_{m_2} \psi_{m_1} \psi_{n_r}^* \cdots \psi_{n_2}^* \psi_{n_1}^*,
\qquad
(r,s \geq 0, n_1 > \cdots n_r \geq 0 > m_s > \cdots > m_1),
\]
and has an isomorphism $\ast \colon \mcF \to \mcF^*$ given by $X \ket{0} \mapsto \bra{0} X^*$.
This isomorphism encodes the particle-hole duality of Fock space, as the role the indices $i_j$ give the positions of the locations without a particle.
We will also use the \defn{shifted vacuum vectors}
\[
\ket{m} = \begin{cases} \psi_{m-1} \dotsm \psi_0 \ket{0} & \text{if } m \geq 0, \\ \psi_m^* \dotsm \psi_{-1}^* \ket{0} & \text{if } m < 0, \end{cases}
\qquad\qquad
\bra{m} = \begin{cases} \bra{0} \psi_0^* \dotsm \psi_{m-1}^* & \text{if } m \geq 0, \\ \bra{0} \psi_{-1} \dotsm \psi_m & \text{if } m < 0. \end{cases}
\]

The \defn{vacuum expectation value} is the unique $\field$-bilinear map
\[
\mcF^* \otimes_{\field} \mcF \to \field,
\qquad
\bra{w} \otimes_{\field} \ket{v}
\mapsto
\braket{w}{v}
\quad
\text{ satisfying }
\quad
\braket{0}{0} = 1,
\qquad
(\bra{w} X) \ket{v} = \bra{w} (X \ket {v}),
\]
for all $X \in \mcA$.
Note the relation $\bra{w} X \ket {v} = \bra{v^\ast} X^* \ket{w^\ast}$, where $\ket{v^\ast}=(\bra{v})^\ast$ and $\bra{w^\ast}=(\ket{w})^\ast$.
We use the abbreviation $\langle X \rangle = \bra{0} X \ket{0}$.
To compute vacuum expectation values, we use \defn{Wick's theorem}.

\begin{thm}[Wick's theorem {\cite{MJD00}}]\label{thm:Wick}
For any finite subsets
$
\{m_1, \dotsc, m_r\}, \{n_1, \dotsc, n_r\} \subseteq \ZZ
$, we have
\[
\langle \psi_{m_r} \dotsm \psi_{m_1} \psi_{n_1}^* \dotsm \psi_{n_r}^* \rangle = \det \bigl[ \langle \psi_{m_i} \psi_{n_j}^* \rangle \bigr]_{i,j=1}^r.
\]
\end{thm}

Next, we define the \defn{current operators} as
\[
a_k := \sum_{i \in \ZZ} \normord{\psi_i \psi_{i+k}^*},
\]
where $\normord{\bullet}$ denotes the normal ordering of free-fermions, which only is important for when $k = 0$ (see, \textit{e.g.}~\cite[Sec.~2]{AZ13} and~\cite[Sec.~5.2]{MJD00}):
\[
\normord{\psi_i \psi_{i+k}^*} \; := \begin{cases}
\psi_i \psi_{i+k}^* & \text{if } i \leq 0, \\
-\psi_{i+k}^* \psi_i & \text{if } i > 0.
\end{cases}
\]
We have $a_k^* = a_{-k}$.
The current operators define the infinite dimensional Heisenberg algebra, which means they satisfy the relations
\begin{equation}
\label{eq:boson_relation}
[a_m, a_k] = m \delta_{m,-k}.
\end{equation}
They also satisfy the additional relations
\[
[a_m, \psi_k] = \psi_{k-m},
\qquad\qquad
[a_m, \psi_k^*] = -\psi_{k+m}^*.
\]
(See, \textit{e.g.},~\cite[Sec.~5.3]{MJD00} for proofs of these relations.)

Next, we define the \defn{Hamiltonian operator} and its supersymmetric analog
\[
H(\xx) := \sum_{k>0} \frac{p_k(\xx)}{k} a_k,
\qquad\qquad
H(\xx/\yy) := \sum_{k>0} \frac{p_k(\xx/\yy)}{k} a_k = H(\xx) - H(\yy).
\]
Note that $-H(\xx/\yy) = H(\yy/\xx)$.
It can be seen the corresponding half-vertex operators satisfy the relations~\cite[Eq.~(17), Eq.~(18)]{Iwao23}
\begin{subequations}
\label{eq:eH_relations}
\begin{align}
\label{eq:eH_commute}
e^{H(\xx/\yy)} \psi_k e^{-H(\xx/\yy)}& = \sum_{i=0}^{\infty} h_i(\xx/\yy) \psi_{k-i},
\\
\label{eq:seH_commute}
e^{-H(\xx/\yy)} \psi^*_k e^{H(\xx/\yy)}& =  \sum_{i=0}^{\infty} h_i(\xx/\yy) \psi^*_{k+i},
\end{align}
\end{subequations}
which are finite if and only if $\xx = \emptyset$ and finitely many $\yy$ are nonzero,
For partitions $\lambda$ and $\mu$ such that $\ell(\lambda), \ell(\mu) \leq \ell$, we define vectors
\[
\bra{\mu} := \bra{-\ell} \psi^*_{\mu_{\ell}-\ell} \cdots \psi^*_{\mu_2-2} \psi^*_{\mu_1-1},
\qquad\qquad
\ket{\lambda} := \psi_{\lambda_1-1} \psi_{\lambda_2-2} \cdots \psi_{\lambda_{\ell}-\ell} \ket{-\ell}.
\]
We remark that
\[
\ket{\lambda, 0} = \ket{\lambda},
\qquad\qquad
\bra{\mu, 0} = \bra{\mu},
\]
where $\ket{\lambda,0}$ means $\ket{(\lambda_1,\dots,\lambda_\ell,0)}$ for $\lambda=(\lambda_1,\dots,\lambda_\ell)$,
and so we can unambiguously consider $\bra{\cdot} \colon \field[\mcP] \to \mcF^*$ and $\ket{\cdot} \colon \field[\mcP] \to \mcF$ as embeddings of algebras.
From Wick's theorem and~\eqref{eq:eH_commute}, we have the following well-known result of free-fermion description of Schur functions, presented here with a plethystic substitution (see, \textit{e.g.}~\cite[Thm.~2.3]{Iwao23}).

\begin{thm}
\label{thm:schur_jacobi_trudi}
Let $\lambda$ be a partition such that $\ell(\lambda) \leq \ell$ and $\mu \subseteq \lambda$.
Then we have
\[
s_{\lambda/\mu}(\xx/\yy) = \bra{\mu} e^{H(\xx/\yy)} \ket{\lambda} = \det \bigl[h_{\lambda_i-\mu_j-i+j}(\xx/\yy) \bigr]_{i,j=1}^{\ell}.
\]
\end{thm}

\begin{remark}
The subspace 
\[
\mcF^0=\{ v_{i_1} \wedge v_{i_2} \wedge \cdots \mid i_1 > i_2 > \cdots, \; i_k = -k\mbox{ for sufficiently large } k\}\subset \mcF
\]
is isomorphic to $\Lambda$ through the linear map $\mcF^0\to \Lambda$ that sends $\ket{v}$ to $\bra{0}e^{H(\xx)}\ket{v}$.
In particular, the set $\{\ket{\lambda}\}_\lambda$ forms a basis of $\mcF^0$ since $\{s_\lambda(\xx)\}_\lambda$ forms a basis of $\Lambda$.
Theorem~\ref{thm:schur_jacobi_trudi} implies that the linear map $\pi_{\mu}:\mcF^0\to\mcF^0$; $\ket{v}\mapsto \ket{\mu}\cdot (\braket{\mu}{v})$ forms a family of projections such that $\mathrm{id}_{\mcF^0}=\sum_{\mu\in \mathcal{P}}\pi_{\mu}$.
Symbolically, we will write $\mathrm{id}=\sum_{\mu}\ket{\mu}\cdot \bra{\mu}$.
\end{remark}

In the sequel, we will use generating functions of the fermionic operators
\[
\psi(z) = \sum_{n \in \ZZ} \psi_n z^n,
\qquad\qquad
\psi^*(z)=\sum_{n \in \ZZ} \psi_n^* z^n.
\]
Note that our convention for $\psi^*(z)$ might differ some some places in the literature by $z = w^{-1}$.

\subsection{Dual and transpose Hamiltonian operators}

In the sequel, we will use other variations of the Hamiltonian operator.
We first define the \defn{(formal) dual Hamiltonian operator}
\[
H^*(\xx / \yy) := \bigl( H(\xx/\yy) \bigr)^* = \sum_{n>0} \frac{p_n(\xx / \yy)}{n} a_{-n}.
\]
Dual to the usual Hamiltonian operator, by applying $\ast$ to~\eqref{eq:eH_relations} we have
\begin{subequations}
\label{eq:eHstar_relations}
\begin{align}
e^{-H^*(\xx/\yy)} \psi^*_k e^{H^*(\xx/\yy)} & = \sum_{i=0}^{\infty} h_i(\xx/\yy) \psi^*_{k-i}, 
\label{eq:seHstar_commute}
\\
e^{H^*(\xx/\yy)} \psi_k e^{-H^*(\xx/\yy)} & = \sum_{i=0}^{\infty} h_i(\xx/\yy) \psi_{k+i}, 
\label{eq:eHstar_commute}
\end{align}
\end{subequations}
where the expansion is finite if and only if only finitely many $\xx$ are nonzero and $\yy = \emptyset$.

To encode the action of $\omega$, we use the \defn{transposed Hamiltonian operator}, which is denoted by
\[
J(\xx/\yy) := \omega H(\xx/\yy) 
= \sum_{k>0} \frac{\omega p_k(\xx/\yy)}{k} a_{k} 
= \sum_{k>0}(-1)^{k-1} \frac{p_k(\xx/\yy)}{k} a_{k}
= -H(-\xx/{-\yy}).
\]
Indeed, we have
\begin{align*}
\bra{\mu} e^{J(\xx/\yy)} \ket{\lambda} & = 
\bra{\mu} \omega e^{H(\xx/\yy)} \ket{\lambda} 
= \bra{\mu} e^{-H(-\xx/-\yy)} \ket{\lambda} 
=
\omega s_{\lambda/\mu}(\xx/\yy) = s_{\lambda'/\mu'}(\xx/\yy)
= s_{\lambda/\mu}(-\yy/{-\xx}).
\end{align*}
We can realize this more fundamentally using free-fermions by realizing the $01$-sequence of $\lambda'$ is equal to that of $\lambda$ in reverse and interchanging $0 \leftrightarrow 1$, and so
\begin{subequations}
\label{eq:conjugate_form}
\begin{align}
\ket{\lambda'} & = (-1)^{\abs{\lambda}} \psi^*_{-\lambda_1} \psi^*_{1-\lambda_2} \dotsm \psi^*_{\ell-1-\lambda_{\ell}} \ket{\ell},
\\
\bra{\mu'} & = (-1)^{\abs{\mu}} \bra{\ell} \psi_{\ell-1-\mu_{\ell}} \psi_{\ell-2-\mu_{\ell-1}} \dotsm \psi_{-\mu_1},
\end{align}
\end{subequations}
(\textit{cf.}~\cite[Prop.~2.1]{AZ13}).
Theorem~\ref{thm:schur_jacobi_trudi} states $s_{\lambda'/\mu'}(\xx/\yy) = \bra{\mu'}e^{H(\xx/\yy)}\ket{\lambda'}$, but by using these alternate descriptions and Wick's theorem, we can obtain the dual Jacobi--Trudi formula for $(-1)^{\abs{\lambda/\mu}} s_{\lambda'/\mu'}(\xx/\yy)$.
This is also the same determinant formula from $\bra{\mu} e^{J(\xx/\yy)} \ket{\lambda}$ by Wick's theorem, which yields the nontrivial relation
\[
\bra{\mu} e^{J(\xx/\yy)} \ket{\lambda} = \bra{\mu'} e^{H(\xx/\yy)} \ket{\lambda'}.
\]
Hence, this identity is equivalent to $s_{\lambda/\mu}(\xx/\yy) = (-1)^{\abs{\lambda/\mu}} s_{\lambda'/\mu'}(\yy/\xx)$.

We also require the following expansions for a single variable $\gamma$:
\begin{equation}
\label{eq:H_expansion}
e^{H(\gamma)} = \sum_{i=0}^{\infty} H_i(a_1, a_2, \ldots) \gamma^i,
\qquad\qquad
e^{J(\gamma)} = \sum_{i=0}^{\infty} E_i(a_1, a_2, \ldots) \gamma^i.
\end{equation}
As a consequence, we have $[e^{H(\gamma)}, a_k] = [e^{J(\gamma)}, a_k] = 0$ for all $k \geq 0$ and $[e^{H(\gamma)},a_{-k}]=\gamma^ke^{H(\gamma)}$.


In order to make the indexing match the usual covariant and contravariant indexing, for a matrix $M = [M_{\lambda}^{\mu}]_{\lambda,\mu \in \mcP}$, we will denote $M^* = [(M^*)_{\mu}^{\lambda} := (M_{\lambda}^{\mu})^*]_{\mu,\lambda\in \mcP}$.
This agrees with thinking of $\ast$ as the transpose conjugate and the $\ast$ duality operator.


\section{Canonical free-fermions}
\label{sec:canonical_fermions}

In this section, we give an orthonormal basis for the space $\mcF^0$ that generalizes the one given in~\cite{Iwao20} and is a special case of the vectors in~\cite{Iwao23}.
For finitely many (noncommutative) expressions $\Phi_1,\dots,\Phi_\ell$, we will use the notation
\[
\prod_{1\leq i\leq \ell}^{\rightarrow} \Phi_i=\Phi_1\Phi_2\cdots \Phi_\ell,\qquad 
\prod_{1\leq i\leq \ell}^{\leftarrow} \Phi_i=\Phi_\ell\cdots \Phi_2\Phi_1
\]
to indicate the order of multiplication.

\subsection{\texorpdfstring{$\bal,\bbe$}{alpha,beta}-deformed vectors and their dual vectors}

We set $A_i=\{-\alpha_1,\dots,-\alpha_i\}$, $B_i=\{\beta_1,\dots,\beta_i\}$ for $i\geq 1$, and $A_i=B_i=\emptyset$ for $i\leq 0$.
For $i\leq j$, we write 
$A_{[i,j]}=A_j\setminus A_{i-1}$, $A_{[i,j)}=A_{[i,j-1]}$, and $A_{(i,j]} = A_{[i+1,j]}$.
It will be convenient to write $A_{[i,j)} = \emptyset / A_{[j,i)}$ and $A_{(i,j]} = \emptyset / A_{(j,i]}$, which admit the formulas
\[
H(A_{[i,j)})+H(A_{[j,k)})=H(A_{[i,k)}),\quad
H(A_{(i,j]})+H(A_{(j,k]})=H(A_{(i,k]})
\]
for arbitrary $i,j,k\in \ZZ$.
If $1\leq i\leq j$, we have $A_{[i,j]}=\{-\alpha_i,-\alpha_{i+1},\dots,-\alpha_j\}$.
We make the analogous definition for $B_{[i,j)}$ and $B_{[i,j]}$ using $B_i$.

Let $\sigma = (\sigma_1, \sigma_2, \dotsc, \sigma_{\ell}) \in \ZZ_{\geq 0}^{\ell}$.
We will use the vectors
\begin{align*}
\ket{\lambda}^{\sigma}_{[\bal,\bbe]} 
& = \prod^{\rightarrow}_{1 \leq i \leq \ell} \left( e^{-H(A_{\sigma_i-1})} \psi_{\lambda_i-i} e^{H(\beta_i)} e^{H(A_{\sigma_i-1})} \right) \ket{-\ell}
%
\end{align*}
and
\begin{align*}
\ket{\lambda}_{\sigma}^{[\bal,\bbe]} 
&=
\prod^{\rightarrow}_{1 \leq i \leq \ell} \left( e^{H^*(A_{\sigma_i})} \psi_{\lambda_i-i} e^{-H^*(\beta_i)} e^{-H^*(A_{\sigma_i})} \right) 
{
e^{H^\ast(A_{\sigma_{\ell}})}}
\ket{-\ell}.
\end{align*}
One can check that these expressions are equivalent to the following equations:
\begin{align*}
\ket{\lambda}^{\sigma}_{[\bal,\bbe]} 
& = 
e^{-H(A_{\sigma_1-1})}
\prod^{\rightarrow}_{1 \leq i \leq \ell} \left( \psi_{\lambda_i-i} e^{H(\beta_i)} 
e^{-H(A_{[\sigma_{i},\sigma_{i+1})})} 
\right) \ket{-\ell}\\
& = 
\prod^{\rightarrow}_{1 \leq i \leq \ell} \left( e^{H(B_{i-1} / A_{\sigma_i-1})} \psi_{\lambda_i-i} e^{-H(B_{i-1} / A_{\sigma_i-1})} \right) \ket{-\ell},
\\ 
%
\ket{\lambda}_{\sigma}^{[\bal,\bbe]} 
&= 
e^{H^\ast(A_{\sigma_\ell}/B_\ell )}
\prod^{\rightarrow}_{1 \leq i \leq \ell} 
 \left(
 e^{H^*(B_{[i,\ell]}/A_{(\sigma_i,\sigma_\ell]})} \psi_{\lambda_i-i} e^{-H^*(B_{[i,\ell]}/A_{(\sigma_i,\sigma_\ell]})}
 \right) \ket{-\ell},
\end{align*}
where the second equality is obtained by using $e^{H(B_\ell)}\ket{-\ell}=\ket{-\ell}$.
{
In most cases, we will assume $\sigma_{\ell}=0$, where we have $e^{H^\ast(A_{\sigma_\ell})}=1$ and
\begin{equation}\label{eq:sigmaell_equals_to_zero}
\begin{aligned}
\ket{\lambda}_{\sigma}^{[\bal,\bbe]} 
&=e^{-H^\ast(B_\ell)}
\prod_{1\leq i\leq \ell}^{\rightarrow}
\left(
e^{H^\ast(A_{\sigma_i}\sqcup B_{[i,\ell]} )}\psi_{\lambda_i-i}e^{-H^\ast(A_{\sigma_i}\sqcup B_{[i,\ell]} )}
\right)\ket{-\ell}.
\end{aligned}
\end{equation}
}

We will frequently take $\sigma = \lambda$, so we will simply write $\ket{\lambda}_{[\bal,\bbe]} := \ket{\lambda}^{\lambda}_{[\bal,\bbe]}$ and $\ket{\lambda}^{[\bal,\bbe]} := \ket{\lambda}_{\lambda}^{[\bal,\bbe]}$.
In the case $\bal = 0$, we will simply write $\ket{\lambda}_{[\bbe]} :=  \ket{\lambda}^\sigma_{[0,\bbe]}$ and $\ket{\lambda}^{[\bbe]} := \ket{\lambda}_\sigma^{[0,\bbe]}$ (note these are independent of $\sigma$), which were introduced in~\cite{Iwao20,Iwao23}.
We also define dual vectors ${}_{[\bal,\bbe]} \bra{\lambda}_{\sigma} := (\ket{\lambda}_{\sigma}^{[\bal,\bbe]})^*$ and ${}^{[\bal,\bbe]}\bra{\lambda}^{\sigma} := (\ket{\lambda}^{\sigma}_{[\bal,\bbe]})^*$, and we use similar shorthands as above such as ${}_{[\bal,\bbe]}\bra{\lambda} := {}_{[\bal,\bbe]}\bra{\lambda}_{\lambda}$.
We will justify calling these dual vectors below in Theorem~\ref{thm:dual_basis}.

In the sequel, we also require the additional vector
\[
{}^{[[\bal,\bbe]]} \bra{\mu} := 
\bra{-\ell}\prod_{1\leq i\leq \ell}^{\leftarrow}
\left(
e^{-H^\ast(B_i/A_{\mu_i})}\psi^\ast_{\mu_i-i}
e^{H^\ast(B_i/A_{\mu_i})}
\right).
\]

\begin{lemma}
\label{lemma:well_defined}
The vectors $\ket{\lambda}_{[\bal,\bbe]}^{\sigma}$, $\ket{\lambda}^{[\bal,\bbe]}_{\sigma}$, and ${}^{[[\bal,\bbe]]} \bra{\mu}$ are well-defined; that is
\[
\ket{\lambda, 0}_{[\bal,\bbe]}^{\widetilde{\sigma}} = \ket{\lambda}_{[\bal,\bbe]}^{\sigma},
\qquad\qquad
\ket{\lambda, 0,0}^{[\bal,\bbe]}_{\overline{\sigma}} = \ket{\lambda,0}^{[\bal,\bbe]}_{\sigma},
\qquad\qquad
{}^{[[\bal,\bbe]]} \bra{\mu,0} = {}^{[[\bal,\bbe]]} \bra{\mu},
\]
where $\widetilde{\sigma} = (\sigma_1, \dotsc, \sigma_{\ell}, \sigma_{\ell+1})$ and $\overline{\sigma} = (\sigma_1, \dotsc, \sigma_{\ell}, \sigma_{\ell+1},\sigma_{\ell+1})$.
\end{lemma}

\begin{proof}
For the vectors $\ket{\lambda}_{[\bal,\bbe]}^{\sigma}$ and ${}^{[[\bal,\bbe]]} \bra{\mu}$, this follows from the fact that $e^{H(\gamma)} \ket{-\ell} = \ket{-\ell}$ by applying the expansion~\eqref{eq:H_expansion} of $e^{H(\gamma)}$ and using $a_i \ket{-\ell} = 0$ for all $i > 0$.
For the vector $\ket{\lambda}^{[\bal,\bbe]}_{\sigma}$, this is~\cite[Lemma~4.6]{Iwao23}.
\end{proof}

Next, we give an analog of~\eqref{eq:conjugate_form}, which will be useful for writing formulas in terms of conjugate shapes.
In particular, this will be important in giving a new (free-fermionic) proof of Theorem~\ref{thm:canonical_conjugate}.

\begin{lemma}
\label{lemma:conjugate_vectors}
Let $\ell\geq \lambda_1,\mu_1$.
Then we have
\begin{align*}
&\ket{\lambda'}_{[\bal,\bbe]}  = (-1)^{\abs{\lambda}} \prod_{1\leq i\leq \ell}^{\rightarrow} \left(e^{H(B_{\lambda_i-1}/A_{i-1})} \psi^*_{i-1-\lambda_i} e^{-H(B_{\lambda_i-1}/A_{i-1})}\right) \ket{\ell},\\
&\ket{\lambda'}^{[\bal,\bbe]} 
=
(-1)^{\abs{\lambda}} 
e^{H^\ast(A_{\ell}/B_{\lambda_\ell})}
\prod_{1\leq i\leq \ell}^{\rightarrow} 
\left(
e^{H^*(B_{[\lambda_i,\lambda_{\ell}]}/A_{(i,\ell]})} \psi^*_{i-1-\lambda_i} e^{-H^*(B_{[\lambda_i,\lambda_{\ell}]}/A_{(i,\ell]})} \right)\ket{\ell},
\\
&{}^{[[\bal,\bbe]]} \bra{\mu'}  = (-1)^{\abs{\lambda}} \bra{\ell} \prod_{1\leq i\leq \ell}^{\leftarrow} \left(e^{H^*(B_{\mu_i}/A_i)} \psi_{i-1-\mu_i} e^{-H^*(B_{\mu_i}/A_i)}\right).
\end{align*}
\end{lemma}

\begin{proof}
We only show the claim for $\ket{\lambda'}_{[\bal,\bbe]}$ as the proofs for $\ket{\lambda'}^{[\bal,\bbe]}$ and ${}^{[[\bal,\bbe]]}\bra{\mu'}$ are similar.
It is sufficient to consider $\ell > \max(\lambda_1, \ell(\lambda))$ by Lemma~\ref{lemma:well_defined}.
Let $X_i$ be a finite set of indeterminates.
To prove the proposition, we use the following commutation relation, which holds for $j_1>j_2>\dots>j_m>k$ and  $|X_i|\leq j_{i}-j_{i+1}$: 
\begin{equation}\label{eq:commutation_relation}
\begin{aligned}
&
\psi_{j_1}e^{H(X_1)}\psi_{j_2}e^{H(X_2)}\cdots e^{H(X_{m-1})}\psi_{j_m}\cdot \psi^\ast_{k}\\
&=
(-1)^m
\left(
e^{H(X_1\sqcup \dots \sqcup X_{m-1})}
\psi^\ast_{k}
e^{-H(X_1\sqcup \dots \sqcup X_{m-1})}
\right)
\psi_{j_1}e^{H(X_1)}\psi_{j_2}e^{H(X_2)}\cdots e^{H(X_{m-1})}\psi_{j_m}.
\end{aligned}
\end{equation}
Equation~\eqref{eq:commutation_relation} follows from the fact that 
$e^{-H(X_1\sqcup \dots \sqcup X_{m-1})}
\psi_{j_1}e^{H(X_1)}\psi_{j_2}e^{H(X_2)}\cdots e^{H(X_{m-1})}\psi_{j_m}$
is a linear combination of $\psi_{j}$ with $j_1\geq j\geq j_m$.

For a partition $\lambda$ and the transpose $\lambda'$, we put
$I=\{\lambda_1-1,\lambda_2-2,\dots,\lambda_\ell-\ell\}$ and $I^\ast=\{-\lambda_1',1-\lambda'_2,\dots,\ell-1-\lambda'_\ell\}$.
They satisfy $I\sqcup I^\ast=\{\ell-1,\ell-2,\dots,-\ell\}$.
Let
\begin{gather*}
X_i=\{\beta_i,\alpha_{\lambda_i-1},\alpha_{\lambda_i-2},\dots,\alpha_{\lambda_{i+1}}\}
=\{\beta_i\}/A_{[\lambda_i,\lambda_{i+1})},\\
\Psi_m=\psi_{\lambda_1-1}e^{H(X_1)}\psi_{\lambda_2-2}e^{H(X_2)}\cdots e^{H(X_{m-1})}\psi_{\lambda_m-m},\\
P_m(k)=e^{H(B_{m-1}/A_{\lambda_m-1})}\psi^\ast_ke^{-H(B_{m-1}/A_{\lambda_m-1})}.
\end{gather*} 
Assume $k\leq \lambda_m-m$. 
We can calculate the vector $\Psi_m\ket{k}$ depending on whether $k\in I$ or $k\in I^\ast$.
If $k\in I$, we have $k=\lambda_m-m$ and 
$\Psi_m\ket{k}=\Psi_{m-1}e^{H(X_{m-1})}\psi_{\lambda_m-m} \ket{k}=\Psi_{m-1}\ket{k+1}$;
otherwise, we have $k<\lambda_m-m$ and 
\[
\Psi_m\ket{k}=
\Psi_m\psi_{k}^\ast\ket{k+1}
=
(-1)^m
\left(
e^{H(A_{\lambda_1-1})}
P_m(k)
e^{-H(A_{\lambda_1-1})}
\right)
\Psi_m\ket{k+1}
\]
from~\eqref{eq:commutation_relation}.
Summarizing, we obtain the equation
\begin{equation}\label{eq:induction_step}
e^{-H(A_{\lambda_1-1})}\Psi_m\ket{k}
=
\begin{cases}
e^{-H(A_{\lambda_1-1})}\Psi_{m-1} \ket{k+1} & (k\in I)\\
(-1)^mP_m(k)e^{-H(A_{\lambda_1-1})}\Psi_{m}\ket{k+1}
&(k\in I^\ast)
\end{cases}.
\end{equation}
Note that, if $k\in I^\ast$ satisfies $\lambda_{m+1}-m-1<k\leq \lambda_m-m$, there uniquely exists some $p$ such that $\lambda_p'=m=p-1-k$.
Using~\eqref{eq:induction_step} repeatedly for $k=-\ell,-\ell+1,\dots,\ell-1$, we obtain
\[
\begin{aligned}
\ket{\lambda}_{[\bal,\bbe]}&=
e^{-H(A_{\lambda_1-1})}\Psi_\ell\ket{-\ell}\\
&
=(-1)^{\lambda_1'+\dots+\lambda_\ell'}
P_{\lambda_1'}(-\lambda'_1)
P_{\lambda_2'}(1-\lambda'_2)
\dots
P_{\lambda_\ell'}(\ell-1-\lambda'_\ell)
e^{-H(A_{\lambda_1-1})}\Psi_0\ket{\ell}\\
&=(-1)^{|\lambda'|}
P_{\lambda_1'}(-\lambda'_1)
P_{\lambda_2'}(1-\lambda'_2)
\dots
P_{\lambda_\ell'}(\ell-1-\lambda'_\ell)
\ket{\ell},
\end{aligned}
\]
which concludes the lemma.
\end{proof}


\begin{prop}
\label{prop:corner_decomp}
We have
\begin{equation}
\label{eq:corner_decomposition}
{}^{[\bal,\bbe]} \bra{\mu} = 
\sum_{\nu} 
\left( \prod_{(i,j) \in \mu/\nu} -(\alpha_i + \beta_j)  \right)
{}^{[[\bal,\bbe]]}\bra{\nu},
\end{equation}
where $\nu\subseteq \mu$ is formed by removing some of the corners of $\mu$.
\end{prop}

\begin{proof}
Let 
\begin{align*}
Q_m(k)
&=e^{-H^\ast(B_{m-1}/A_{k-1})}\psi^\ast_{k-m}e^{H^\ast(B_{m-1}/A_{k-1})},\\
R_m(k)
&=e^{-H^\ast(B_{m}/A_{k})}\psi^\ast_{k-m}e^{H^\ast(B_{m}/A_{k})}
=
e^{-H^\ast(\beta_m)}e^{H^\ast(-\alpha_k)}
Q_m(k)
e^{-H^\ast(-\alpha_k)}e^{H^\ast(\beta_m)}\\
&=Q_{m+1}(k+1).
\end{align*}
We begin by noting the relation
\begin{equation}
\label{eq:rewrite_star}
\psi_{k}^* = e^{H^*(-\alpha_{k})} \psi^*_{k} e^{-H^*(-\alpha_{k})} - \alpha_{k} \psi^*_{k-1},
\end{equation}
by rewriting~\eqref{eq:seHstar_commute}.
By using~\eqref{eq:rewrite_star} and~\eqref{eq:seHstar_commute}, we have the identity
\begin{align*}
 \psi^*_{k} & = 
e^{-H^*(\beta_m)} \left(\psi^*_{k} - \beta_m \psi^*_{k-1}\right)e^{H^*(\beta_m)}
\\ & =
e^{-H^*(\beta_m)}
\left(
e^{H^*(-\alpha_{k})} \psi^*_{k} e^{-H^*(-\alpha_{k})} - (\alpha_{k} + \beta_m) \psi^*_{k-1}
\right)e^{H^*(\beta_m)},
\end{align*}
which yields 
\begin{equation}\label{eq:rec}
Q_m(k)=R_m(k)-(\alpha_{k} + \beta_m)R_m(k-1).
\end{equation}
Since $Q_{m+1}(k)R_m(k-1)=Q_{m+1}(k)^2=0$, we have $Q_{m+1}(k)Q_{m}(k)=Q_{m+1}(k)R_{m}(k)$.
By using $\bra{-\ell}R_\ell(-1)=0$ and~\eqref{eq:rec} repeatedly for $m=1,2,\dots,\ell$, we obtain
\begin{align*}
{}^{[\bal,\bbe]}\bra{\mu} 
&=
\bra{-\ell}Q_\ell(\mu_\ell)\cdots Q_{2}(\mu_2)Q_1(\mu_1)
\\
&=
\sum_{
\substack{
(\nu_1,\dots,\nu_\ell):\\
\mu_i-\nu_i\in \{0,1\}\\
\nu_i\geq \mu_{i+1}
}
}
\left(
\prod_{m:\,\nu_m=\mu_m-1}
-(\alpha_{\mu_m} + \beta_m) 
\right)
\bra{-\ell}R_\ell(\nu_\ell)\cdots R_{2}(\nu_2)R_1(\nu_1)
\\
&=
\sum_{\nu}
\left(
\prod_{(i,j) \in \mu/\nu} -(\alpha_i + \beta_j) 
\right)
{}^{[[\bal,\bbe]]}\bra{\nu},
\end{align*}
where $\nu\subseteq \mu$ is formed by removing some of the corners of $\mu$.
\end{proof}

\begin{remark}
\label{rem:multischur}
The vector $\ket{\lambda}_{[\bal, \bbe]}$ is equal to $\ket{\lambda}_{\xx/\yy}$ from~\cite[Eq.~(24)]{Iwao23} with $x^{(i)} = B_{i-1}$ and $y^{(i)} = A_{\lambda_i-1}$.
As a consequence, we have that the (skew) dual canonical Grothendieck functions are (skew) multiSchur functions as defined by Lascoux~\cite{Lascoux03}.
\end{remark}

\begin{prop}
\label{prop:corner_decomp_inverse}
We have
\begin{equation}
\label{eq:inverse_corner}
{}^{[[\bal,\bbe]]} \bra{\mu} = 
\sum_{\nu \subseteq \mu} 
\left( \prod_{(i,j) \in \mu/\nu} (\alpha_i + \beta_j)  \right)
{}^{[\bal,\bbe]}\bra{\nu}.
\end{equation}
\end{prop}

\begin{proof}
The proof is similar to the proof of Proposition~\ref{prop:corner_decomp}.

Alternatively, we could prove this by applying Equation~\eqref{eq:corner_decomposition} to the right hand side.
\end{proof}

\subsection{Duality of \texorpdfstring{$\bal,\bbe$}{alpha,beta}-deformed vectors}

We will now prove that the $\bal,\bbe$-deformed vectors are dual bases (Theorem~\ref{thm:dual_basis}).

We give the rectification lemma, a key computational tool in showing how to transform vectors for a general sequence $\vv = (v_1, v_2, \ldots, v_{\ell}) \in \ZZ_{\geq 0}^{\ell}$ as
\[
\ket{\vv}^{\sigma}_{[\bal,\bbe]} = \sum_{\lambda} r_{\vv}(\lambda) \ket{\lambda}_{[\bal,\bbe]},
\qquad\qquad
\ket{\vv}_{\sigma}^{[\bal,\bbe]} = \sum_{\lambda} r^{\vv}(\lambda) \ket{\lambda}^{[\bal,\bbe]}.
\]
In particular, we obtain the polynomials $r_{\vv}(\lambda), r^{\vv}(\lambda)$ by sorting the sequence $\vv$.

\begin{lemma}[Rectification lemma]
\label{lemma:rectification}
If $k \leq j$, then
\begin{align*}
\psi_k e^{H(\gamma)} \psi_j & = \sum_{i=k}^j \gamma^{j-i+1} \psi_i e^{H(\gamma)} \psi_{k-1} - \sum_{i=k+1}^j \gamma^{j-i} \psi_i e^{H(\gamma)} \psi_k,
\\
\psi_k e^{-H^*(\gamma)} \psi_j & = \sum_{i=k}^j \gamma^{i-k+1} \psi_{j+1} e^{-H^*(\gamma)} \psi_i - \sum_{i=k}^{j-1} \gamma^{i-k} \psi_j e^{-H^*(\gamma)} \psi_i,
\end{align*}
\end{lemma}

\begin{proof}
For the first equality, using~\eqref{eq:eH_commute} and $\psi_k \psi_k = 0$, we have
\begin{align*}
\psi_k e^{H(\gamma)} \psi_j & =  \sum_{i=0}^{\infty} \psi_k\gamma^i \psi_{j-i} e^{H(\gamma)}
\allowdisplaybreaks \\
& = -\sum_{i=0}^{j-k-1} \gamma^i \psi_{j-i} e^{H(\gamma)} \psi_k + \sum_{i=0}^{j-k-1} \gamma^{i+1} \psi_{j-i} e^{H(\gamma)} \psi_{k-1} + \gamma^{j-k+1} \psi_k e^{H(\gamma)} \psi_{k-1},
\end{align*}
which yields the claim after reindexing the sums.
The second equality is similarly computed but using~\eqref{eq:eHstar_commute}:
\begin{align*}
\psi_k e^{-H^*(\gamma)} \psi_j & = \sum_{i=0}^{\infty} e^{-H^*(\gamma)} \gamma^i \psi_{k+i} \psi_j
\allowdisplaybreaks \\
& = -\sum_{i=0}^{j-k-1} \gamma^i (\psi_j - \gamma \psi_{j+1}) e^{-H^*(\gamma)} \psi_{k+i} + \gamma^{j-k+1} \psi_{j+1} e^{-H^*(\gamma)} \psi_j,
\end{align*}
which is the claim after splitting the sum.
\end{proof}

A consequence of Lemma~\ref{lemma:rectification} is a generalization of~\cite[Cor.~2.10]{Iwao20}, which is the special case when $j = k$.

\begin{lemma}
\label{lemma:reduction}
We have
\[
\ket{\lambda}_{[\bal,\bbe]} = \sum_{\mu \subseteq \lambda} b_{\lambda}^{\mu} \ket{\mu}_{[\bbe]},
\qquad\qquad
\ket{\lambda}^{[\bal,\bbe]} = \sum_{\lambda \subseteq \mu} B_{\lambda}^{\mu} \ket{\mu}^{[\bbe]},
\]
where $b_{\lambda}^{\mu}, B_{\lambda}^{\mu} \in \ZZ[\bal, \bbe]$ with $b_{\lambda}^{\lambda} = B_{\lambda}^{\lambda} = 1$.
\end{lemma}

\begin{proof}
Applying the relations~\eqref{eq:eH_commute} and~\eqref{eq:eHstar_commute}, we have
\begin{align*}
\ket{\lambda}_{[\bal,\bbe]} & 
=
\prod^{\rightarrow}_{1 \leq i \leq \ell} \left( e^{H(B_{i-1} / A_{\lambda_i-1})} \psi_{\lambda_i-i} e^{-H(B_{i-1} / A_{\lambda_i-1})} \right) \ket{-\ell}\\
&=
\prod^{\rightarrow}_{1 \leq i \leq \ell} 
\sum_{p=0}^\infty
\left( 
(-1)^{p}e_{p}(A_{\lambda_i-1})\cdot 
e^{H(B_{i-1})} \psi_{\lambda_i-i-p} e^{-H(B_{i-1})} 
\right) \ket{-\ell}\\
&=
\sum_{p_1,\dots,p_\ell=0}^\infty
\prod^{\rightarrow}_{1 \leq i \leq \ell} 
\left( 
(-1)^{p_i}e_{p_i}(A_{\lambda_i-1})\cdot 
e^{H(B_{i-1})} \psi_{\lambda_i-i-p_i} e^{-H(B_{i-1})} 
\right) \ket{-\ell}\\
& = \sum_{p_1,\dotsc,p_{\ell}=0}^{\infty} \prod_{i=1}^{\ell} (-1)^{p_i} e_{p_i}(A_{\lambda_i-1}) \ket{\lambda - (p_1, \dotsc, p_{\ell})}_{[\bbe]},
\\ \ket{\lambda}^{[\bal,\bbe]} & = \sum_{p_1,\dotsc,p_{\ell}=0}^{\infty} \prod_{i=1}^{\ell} h_{p_i}(A_{\lambda_i}) \ket{\lambda + (p_1, \dotsc, p_{\ell})}^{[\bbe]}.
\end{align*}
The claim then follows from Lemma~\ref{lemma:rectification}.
\end{proof}

We require the following dual basis theorem for the special case of $\bal = 0$.

\begin{thm}[{\cite[Thm.~3.2]{Iwao23}}]
We have
\begin{equation}
\label{eq:refined_duality}
{}_{[\bbe]}\braket{\mu}{\lambda}_{[\bbe]} = {}^{[\bbe]}\braket{\mu}{\lambda}^{[\bbe]} = \delta_{\lambda\mu}.
\end{equation}
\end{thm}

\begin{lemma}
\label{lemma:refined_expand_det}
For $\ell \geq \ell(\lambda), \ell(\mu)$ and $\ell' \geq \lambda_1, \mu_1$, we have
\begin{subequations}
\begin{align}
b_{\lambda}^{\mu} & = \det \bigl[ h_{\lambda_i-\mu_j-i+j}(B_{[j,i)} / A_{\lambda_i-1}) \bigr]_{i,j=1}^{\ell}
\\ & = (-1)^{\abs{\lambda} - \abs{\mu}} \det \bigl[ h_{\lambda'_i-\mu'_j-i+j}(A_{i-1} / B_{(\mu'_j,\lambda'_i)}) \bigr]_{i,j=1}^{\ell'}, \label{eq:dualb_det_conjugate}
\\
B_{\lambda}^{\mu} & = \det \bigl[ h_{-\lambda_i+\mu_j+i-j}(A_{\lambda_i} / B_{[j,i)}) \bigr]_{i,j=1}^{\ell} \label{eq:b_det}
\\ & = (-1)^{\abs{\mu} - \abs{\lambda}} \det \bigl[ h_{-\lambda'_i + \mu'_j + i - j}(B_{[\mu'_j,\lambda'_i]} / A_{i-1}) \bigr]_{i,j=1}^{\ell'}. \label{eq:b_det_conjugate}
\end{align}
\end{subequations}
\end{lemma}

\begin{proof}
We only show~\eqref{eq:b_det} for $B_{\lambda}^{\mu}$ as the others are similar (see~\cite[Ex.~3.6]{Iwao23} with Remark~\ref{rem:multischur}).
From~\eqref{eq:refined_duality}, we can write
\[
\id = \sum_{\mu} \ket{\mu}_{[\bbe]} \cdot {}_{[\bbe]} \bra{\mu} = \sum_{\mu} \ket{\mu}^{[\bbe]} \cdot {}^{[\bbe]} \bra{\mu}.
\]
Hence, we have
\[
\ket{\lambda}^{[\bal,\bbe]} = \sum_{\mu} \ket{\mu}^{[\bbe]} \cdot {}^{[\bbe]} \braket{\mu}{\lambda}^{[\bal,\bbe]} = \sum_{\mu} B_{\lambda}^{\mu} \ket{\mu}^{[\bbe]},
\]
and pairing this with ${}^{[\bbe]} \bra{\mu}$, we see that $B_{\lambda}^{\mu} = {}^{[\bbe]} \braket{\mu}{\lambda}^{[\bal,\bbe]}$.
Next, note that for
\begin{align*}
P_i & = 
{
 e^{H^*(B_{[i,\ell]}/A_{(\lambda_i,\lambda_\ell]})} \psi_{\lambda_i-i} e^{-H^*(B_{[i,\ell]}/A_{(\lambda_i,\lambda_\ell]})}
 }
=
\sum_{k=0}^{\infty} h_k(B_{[i,\ell]}/A_{(\lambda_i,\lambda_\ell]}) \psi_{\lambda_i-i+k},
\\
Q_j & = 
{
e^{-H^*(A_{\lambda_{\ell}}/B_{[j,\ell]})} \psi^*_{\mu_j-j} e^{H^*(A_{\lambda_{\ell}}/B_{[j,\ell]})} 
}
= \sum_{m=0}^{\infty} h_m(A_{\lambda_{\ell}}/B_{[j,\ell]}) \psi^*_{\mu_j-j-m},
\end{align*}
we apply Wick's theorem to obtain
\begin{align*}
{}^{[\bbe]} \braket{\mu}{\lambda}^{[\bal,\bbe]}
& = \bra{-\ell} Q_{\ell} \cdots Q_1 P_1 \cdots P_{\ell} \ket{-\ell}
\\ & = \det \bigl[ \bra{-\ell} Q_j P_i \ket{-\ell} \bigr]_{i,j=1}^{\ell}
\\ & = \det \left[ \sum_{k,m=0}^{\infty} h_m(A_{\lambda_{\ell}}/B_{[j,\ell]}) h_k(B_{[i,\ell]}/A_{(\lambda_i,\lambda_\ell]}) \bra{-\ell} \psi^*_{\mu_j-j-m} \psi_{\lambda_i-i+k} \ket{-\ell} \right]_{i,j=1}^{\ell}
\\ & = \det \left[ \sum_{k=0}^{\infty} h_{-\lambda_i+\mu_j+i-j-k}(A_{\lambda_{\ell}}/B_{[j,\ell]}) h_k(B_{[i,\ell]}/A_{(\lambda_i,\lambda_\ell]})  \right]_{i,j=1}^{\ell}
\\ & = \det \bigl[  h_{-\lambda_i+\mu_j+i-j}(A_{\lambda_i} /
B_{[j,i)}) \bigr]_{i,j=1}^{\ell}
\end{align*}
as desired.
\end{proof}

Lemma~\ref{lemma:reduction} implies that $b_{\lambda}^{\mu} = 0$ and $B_{\mu}^{\lambda} = 0$ whenever $\mu \not\subseteq \lambda$.

\begin{thm}
\label{thm:dual_basis}
We have
\[
{}_{[\bal,\bbe]}\braket{\mu}{\lambda}_{[\bal,\bbe]} = {}^{[\bal,\bbe]}\braket{\mu}{\lambda}^{[\bal,\bbe]} = \delta_{\lambda\mu}.
\]
\end{thm}

\begin{proof}
Lemma~\ref{lemma:reduction} implies that
\begin{align*}
{}^{[\bal,\bbe]} \bra{\mu} & 
= \sum_{\nu \subseteq \mu} {}^{[\bbe]} \bra{\nu} (b^*)_{\nu}^{\mu},
& 
{}_{[\bal,\bbe]} \bra{\mu} & 
= \sum_{\mu \subseteq \nu} {}_{[\bbe]} \bra{\nu} (B^*)_{\nu}^{\mu},
\end{align*}
by applying $\ast$.
This with~\eqref{eq:refined_duality} implies that it is sufficient to show the claim when either $\lambda \subseteq \mu$ for ${}^{[\bal,\bbe]}\braket{\mu}{\lambda}^{[\bal,\bbe]}$ or $\mu \subseteq \lambda$ for ${}_{[\bal,\bbe]}\braket{\mu}{\lambda}_{[\bal,\bbe]}$.

We note that it is sufficient to show ${}_{[\bal,\bbe]}\braket{\mu}{\lambda}_{[\bal,\bbe]} = \delta_{\lambda\mu}$ as the other equality is formed by applying $\ast$.
We can write the pairing using
\begin{align*}
P_i & = e^{-H(A_{\lambda_i-1})} e^{H(B_{i-1})} \psi_{\lambda_i-i} e^{-H(B_{i-1})} e^{H(A_{\lambda_i-1})} = \sum_{k=0}^{\infty} h_k(B_{i-1}/A_{\lambda_i-1}) \psi_{\lambda_i-i-k},
\\
Q_j & = e^{-H(A_{\mu_j})} e^{H(B_{j-1})} \psi^*_{\mu_j-j} e^{-H(B_{j-1})} e^{H(A_{\mu_j})} = \sum_{m=0}^{\infty} h_m(A_{\mu_j}/B_{j-1}) \psi^*_{\mu_j-j+m},
\end{align*}
and then we apply Wick's theorem to obtain
\begin{align*}
{}_{[\bal,\bbe]} \braket{\mu}{\lambda}_{[\bal,\bbe]}
& = \bra{-\ell} Q_{\ell} \cdots Q_1 P_1 \cdots P_{\ell} \ket{-\ell}
\\ & = \det \bigl[ \bra{-\ell} Q_j P_i \ket{-\ell} \bigr]_{i,j=1}^{\ell}
\\ & = \det \left[ \sum_{k,m=0}^{\infty} h_k(B_{i-1}/A_{\lambda_i-1}) h_m(A_{\mu_j}/B_{j-1}) \bra{-\ell} \psi^*_{\mu_j-j+m} \psi_{\lambda_i-i-k} \ket{-\ell} \right]_{i,j=1}^{\ell}
\\ & = \det \left[ \sum_{m=0}^{\infty} h_{\lambda_i-\mu_j-i+j-m}(B_{i-1}/A_{\lambda_i-1}) h_m(A_{\mu_j}/B_{j-1}) \right]_{i,j=1}^{\ell}
\\ & = \det \bigl[ h_{\lambda_i-\mu_j-i+j}((A_{\mu_j} \sqcup B_{i-1})/(A_{\lambda_i-1} \sqcup B_{j-1})) \bigr]_{i,j=1}^{\ell}.
\end{align*}

Now we need to examine the matrix
\begin{equation}
\label{eq:big_h_matrix}
\mcH = \bigl[ h_{\lambda_i-\mu_j-i+j}(B_{[j,i)} / A_{(\mu_j, \lambda_i)}) \bigr]_{i,j=1}^{\ell}.
\end{equation}
If $i \leq j$, then we have $\lambda_i \geq \lambda_j \geq \mu_j$, and we set $m = \lambda_i-\mu_j-i+j$.
Thus, we have
\[
h_m(B_{[j,i)} / A_{(\mu_j, \lambda_i)}) = (-1)^m e_m(A_{(\mu_j, \lambda_i)} \sqcup B_{[i,j)}),
\]
which if $i < j$, this is equal to $0$ since there are only $\lambda_i - \mu_j - 1 + j - i < m$ variables.
Thus, the matrix $\mcH$ is triangular, so the determinant is equal to the product of the diagonal entries.
Hence, the determinant equals $\delta_{\lambda\mu}$ as desired.
\end{proof}

\begin{remark}
\label{rem:similar_proof}
Our proof is essentially the same as~\cite[Thm.~3.5]{HJKSS21} in the language of free-fermions.
Note that the matrix $\mcH$ in~\eqref{eq:big_h_matrix} is the matrix of~\cite[Eq.~(3.10)]{HJKSS21}.
The Cauchy--Binet analog~\cite[Lemma~3.1]{HJKSS21} is being played by Wick's theorem.
The reduction using~\cite[Lemma~2.7]{HJKSS21} can be given by the analog of~(III) shown in the proof~\cite[Thm~3.2]{Iwao23} (with a similar proof), but as our proof indicates this step is not necessary.
\end{remark}

\section{Grothendieck polynomials and algebraic formulas}
\label{sec:results}

In this section, we show our main result (Theorem~\ref{thm:jacobi_trudi}, Corollary~\ref{cor:free_fermion_grothendiecks}): the (skew) (dual) canonical Grothendieck polynomials defined in~\cite{HJKSS21} (we will drop the word ``refined'' here for simplicity) can be given as matrix elements using the fermionic Fock space.
From this description, we give a number of algebraic identities as a consequence of the fermionic Fock space.

\subsection{Jacobi--Trudi formulas}

We show our main result by applying Wick's theorem and showing that the resulting determinants are equal to the Jacobi--Trudi formulas from~\cite[Thm.~1.7]{HJKSS21}.

\begin{thm}[Jacobi--Trudi formulas]
\label{thm:jacobi_trudi}
For $\ell \geq \ell(\lambda)$ and $\ell' \geq \lambda_1$, we have
\begin{align*}
{}^{[[\bal,\bbe]]} \bra{\mu} e^{H(\xx_n)} \ket{\lambda}^{[\bal,\bbe]} & = C \det \bigl[ h_{\lambda_i-\mu_j-i+j}\bigl(\xx_n \ds (A_{(\mu_j,\lambda_i]} \sqcup B_{[i,j]}) \bigr) \bigr]_{i,j=1}^{\ell}
\\ & = C' \det \bigl[ e_{\lambda'_i-\mu'_j-i+j}\bigl(\xx_n \ds (A_{(j,i)} \sqcup B_{(\lambda'_i,\mu'_j]}) \bigr) \bigr]_{i,j=1}^{\ell'},
\allowdisplaybreaks
\\ {}^{[\bal,\bbe]} \bra{\mu} e^{H(\xx_n)} \ket{\lambda}^{[\bal,\bbe]} & = C \det \bigl[ h_{\lambda_i-\mu_j-i+j}\bigl(\xx_n \ds (A_{[\mu_j,\lambda_i]} \sqcup B_{[i,j)}) \bigr) \bigr]_{i,j=1}^{\ell}
\\ & = C' \det \bigl[ e_{\lambda'_i-\mu'_j-i+j}\bigl(\xx_n \ds (A_{[j,i)} \sqcup B_{(\lambda'_i,\mu'_j)}) \bigr) \bigr]_{i,j=1}^{\ell'},
\allowdisplaybreaks
\\ {}_{[\bal,\bbe]} \bra{\mu} e^{H(\xx_n)} \ket{\lambda}_{[\bal,\bbe]} & = \det \bigl[ h_{\lambda_i-\mu_j-i+j}(\xx_n \sqcup A_{[\lambda_i,\mu_j]} \sqcup B_{[j,i)}) \bigr]_{i,j=1}^{\ell}
\\ & = \det \bigl[ e_{\lambda'_i-\mu'_j-i+j}(\xx_n \sqcup A_{[i,j)} \sqcup B_{(\mu'_j,\lambda'_i)}) \bigr]_{i,j=1}^{\ell'},
\end{align*}
where
\[
C = \prod_{i=1}^{\ell} \prod_{j=1}^n (1 - \beta_i x_j),
\qquad\qquad
C' = \prod_{i=1}^{\ell'} \prod_{j=1}^n (1 + \alpha_i x_j)^{-1}.
\]
\end{thm}

\begin{proof}
We first rewrite our matrix elements using Wick's theorem in the following form
\[
\bra{-\ell} Q_{\ell} \cdots Q_1 M e^{H(\xx_n)} M^{-1} P_1 \cdots P_{\ell} \ket{-\ell} = C_M \det \left( \bra{-\ell} Q_j e^{H(\xx_n)} P_i \ket{-\ell} \right)_{i,j=1}^{\ell},
\]
where $Me^{H(\xx_n)}M^{-1} = C_Me^{H(\xx_n)}$, for the $h$ versions and similarly for the dual $e$ versions but using the vectors from Lemma~\ref{lemma:conjugate_vectors}.
The rest of the proof is similar to the computation for the proof of Theorem~\ref{thm:dual_basis}.

For ${}^{[[\bal,\bbe]]} \bra{\mu} e^{H(\xx_n)} \ket{\lambda}^{[\bal,\bbe]}$, we use $M = e^{H^*(B_\ell)}$ and
\begin{align*}
P_i & = e^{H^*(A_{\lambda_i})} e^{H^*(B_{[i,\ell]})} \psi_{\lambda_i-i} e^{-H^*(B_{[i,\ell]})} e^{-H^*(A_{\lambda_i})},
\\
Q_j & = e^{H^*(A_{\mu_j})} e^{H^*(B_{(j,\ell]})} \psi^*_{\mu_j-j} e^{-H^*(B_{(j,\ell]})} e^{-H^*(A_{\mu_j})}.
\end{align*}

For ${}^{[\bal,\bbe]} \bra{\mu} e^{H(\xx_n)} \ket{\lambda}^{[\bal,\bbe]}$, we use $M = e^{H^*(B_\ell)}$ and
\begin{align*}
P_i & = e^{H^*(A_{\lambda_i})} e^{H^*(B_{[i,\ell]})} \psi_{\lambda_i-i} e^{-H^*(B_{[i,\ell]})} e^{-H^*(A_{\lambda_i})},
\\
Q_j & = e^{H^*(A_{\mu_j-1})} e^{H^*(B_{[j,\ell]})} \psi^*_{\mu_j-j} e^{-H^*(B_{[j,\ell]})} e^{-H^*(A_{\mu_j-1})}.
\end{align*}

For ${}_{[\bal,\bbe]} \bra{\mu} e^{H(\xx_n)} \ket{\lambda}_{[\bal,\bbe]}$, we use $M = e^{H(B_\ell)}$ and
\begin{align*}
P_i & = e^{-H(A_{\lambda_i-1})} e^{-H(B_{[i,\ell]})} \psi_{\lambda_i-i} e^{H(B_{[i,\ell]})} e^{H(A_{\lambda_i-1})},
\\
Q_j & = e^{-H(A_{\mu_j})} e^{-H(B_{(j,\ell]})} \psi^*_{\mu_j-j} e^{H(B_{(j,\ell]})} e^{H(A_{\mu_j})}.
\end{align*}

To compute $C_M$, for $\G_{\lambda/\mu}(\xx_n; \bal, \bbe)$ and $\G_{\lambda\ds\mu}(\xx_n; \bal, \bbe)$ we use
\begin{equation}
\label{eq:conjugate_op_star}
e^{H^*(B_\ell)} e^{H(\xx)} e^{-H^*(B_\ell)} = \prod_{j=1}^\ell \prod_{k=1}^{\infty} (1-\beta_j x_k) e^{H(\xx)},
\end{equation}
and we use $e^{H(B_\ell)} e^{H(\xx_n)} e^{-H(B_\ell)}=e^{H(\xx_n)}$ for $\dG_{\lambda/\mu}(\xx_n; \bal, \bbe)$.
\end{proof}

\begin{dfn}
We define the (skew) (refined) canonical Grothendieck functions as
\begin{subequations}
\label{eq:grothendieck_defn}
\begin{align}
\G_{\lambda/\mu}(\xx_n; \bal, \bbe) & := {}^{[[\bal,\bbe]]} \bra{\mu} e^{H(\xx_n)} \ket{\lambda}^{[\bal,\bbe]},
\\
\dG_{\lambda/\mu}(\xx_n; \bal, \bbe) & :={}_{[\bal,\bbe]} \bra{\mu} e^{H(\xx_n)} \ket{\lambda}_{[\bal,\bbe]},
\\
\G_{\lambda\ds\mu}(\xx_n; \bal, \bbe) & := {}^{[\bal,\bbe]} \bra{\mu} e^{H(\xx_n)} \ket{\lambda}^{[\bal,\bbe]}.
\end{align}
\end{subequations}
\end{dfn}

Theorem~\ref{thm:jacobi_trudi} can be rephrased as follows.

\begin{cor}
\label{cor:free_fermion_grothendiecks}
The symmetric functions $\G_{\lambda/\mu}(\xx_n; \bal, \bbe)$ and $\dG_{\lambda/\mu}(\xx_n; \bal, \bbe)$ are equal to those in~\cite[Thm.~1.7]{HJKSS21}.
\end{cor}

Note that the middle two determinants in Theorem~\ref{thm:jacobi_trudi} do not appear in~\cite{HJKSS21}.

Equation~\eqref{eq:corner_decomposition} gives the refined analog of $G_{\lambda\ds\mu}(\xx_n; 0, -1)$ in~\cite[Eq.~(6.4)]{Buch02} and $\G_{\lambda\ds\mu}(\xx; \alpha, \beta)$ in~\cite[Prop.~8.8]{Yel17}, which is also the generalization of~\cite[Prop.~4.7]{Iwao20}:
\[
\G_{\lambda\ds\mu}(\xx; \bal, \bbe) = \sum_{\nu} \left( \prod_{(i,j) \in \mu/\nu} -(\alpha_i + \beta_j)  \right) \G_{\lambda/\nu}(\xx; \bal, \bbe),
\]
where $\nu\subseteq \mu$ is formed by removing some of the corners of $\mu$.
Subsequently, Equation~\eqref{eq:inverse_corner} yields the generalization of~\cite[Eq.~(7.4)]{Buch02} (which is related to the previous formula by M\"obius inversion):
\[
\G_{\lambda/\mu}(\xx; \bal, \bbe) = \sum_{\nu \subseteq \mu} \left( \prod_{(i,j) \in \mu/\nu} (\alpha_i + \beta_j)  \right) \G_{\lambda\ds\nu}(\xx; \bal, \bbe).
\]

We require the reformulations
\[
\G_{\lambda \ds \mu}(\xx; \bal, \bbe) = {}_{[\bal,\bbe]} \bra{\lambda} e^{H^*(\xx)} \ket{\mu}_{[\bal,\bbe]},
\qquad\qquad
\dG_{\lambda / \mu}(\xx; \bal, \bbe) = {}^{[\bal,\bbe]} \bra{\lambda} e^{H^*(\xx)} \ket{\mu}^{[\bal,\bbe]},
\]
by applying $\ast$ to our formulas in the sequel.

We can obtain a Jacobi--Trudi formula for $\omega \G_{\lambda/\mu}(\xx; \bal, \bbe)$ (Theorem~\ref{thm:canonical_conjugate}) by applying $\omega$ to each entry of the matrix since $\omega$ is an algebra morphism.
By a direct comparison, we have the following.

\begin{thm}[{\cite[Thm.~1.8]{HJKSS21}}]
\label{thm:canonical_conjugate}
We have
\[
\omega \G_{\lambda/\mu}(\xx; \bal, \bbe) = \G_{\lambda'/\mu'}(\xx; \bbe, \bal),
\qquad\qquad
\omega \dG_{\lambda/\mu}(\xx; \bal, \bbe) = \dG_{\lambda'/\mu'}(\xx; \bbe, \bal).
\]
\end{thm}

We remark that Theorem~\ref{thm:canonical_conjugate} can be reformulated as
\[
\G_{\lambda'/\mu'}(\xx; \bbe, \bal) = {}^{[[\bal,\bbe]]} \bra{\mu} e^{J(\xx)} \ket{\lambda}^{[\bal,\bbe]},
\qquad\qquad
\dG_{\lambda'/\mu'}(\xx; \bbe, \bal) ={}_{[\bal,\bbe]} \bra{\mu} e^{J(\xx)} \ket{\lambda}_{[\bal,\bbe]}.
\]

\subsection{Skew Schur expansions}

We obtain~\cite[Thm.~8.7, 8.8]{HJKSS21} by using the identity operator and Wick's theorem.
Like Remark~\ref{rem:similar_proof}, this is fundamentally the same proof in the language of free-fermions.
However, we obtain slightly different formulas than~\cite[Thm.~8.2, 8.3]{HJKSS21} for $\G_{\lambda/\mu}(\xx; \bal, \bbe)$.
We also have the similar result for $\G_{\lambda\ds\mu}(\xx; \bal, \bbe)$, which has not been stated previously in the literature.
In particular, we compute the following matrix elements.


\begin{thm}
\label{thm:matrix_elements}
Let $\ell \geq \ell(\lambda)$.
We have
\begin{align*}
&\G_{\lambda/\mu}(\xx; \bal, \bbe)  = \sum_{\eta \subseteq \mu \subseteq \lambda \subseteq \nu} \mcD_{\eta}^{\mu}(\bal, \bbe) s_{\nu/\eta}(\xx) \mcI_{\lambda}^{\nu}(\bal, \bbe),
\allowdisplaybreaks \\	
&\G_{\lambda'/\mu'}(\xx; \bal, \bbe)  = \sum_{\eta \subseteq \mu \subseteq \lambda \subseteq \nu} \widetilde{\mcD}_{\eta}^{\mu}(\bal, \bbe) s_{\nu'/\eta'}(\xx) \widetilde{\mcI}_{\lambda}^{\nu}(\bal, \bbe),
\allowdisplaybreaks \\
&\G_{\lambda\ds\mu}(\xx; \bal, \bbe)  = \sum_{\eta \subseteq \mu \subseteq \lambda \subseteq \nu} (\mcE^*)_{\eta}^{\mu}(\bal, \bbe) s_{\nu/\eta}(\xx) \mcI_{\lambda}^{\nu}(\bal, \bbe),
\allowdisplaybreaks \\	
&\G_{\lambda'\ds\mu'}(\xx; \bal, \bbe)  = \sum_{\eta \subseteq \mu \subseteq \lambda \subseteq \nu} (\widetilde{\mcE}^*)_{\eta}^{\mu}(\bal, \bbe) s_{\nu'/\eta'}(\xx) \widetilde{\mcI}_{\lambda}^{\nu}(\bal, \bbe),
\allowdisplaybreaks \\
&\dG_{\lambda/\mu}(\xx; \bal, \bbe)  = \sum_{\mu \subseteq \eta \subseteq \nu \subseteq \lambda} (\mcI^*)_{\eta}^{\mu}(\bal, \bbe) s_{\nu/\eta}(\xx) \mcE^{\nu}_{\lambda}(\bal, \bbe),
\allowdisplaybreaks \\
&\dG_{\lambda'/\mu'}(\xx; \bal, \bbe)  = \sum_{\mu \subseteq \eta \subseteq \nu \subseteq \lambda} (\widetilde{\mcI}^*)_{\eta}^{\mu}(\bal, \bbe) s_{\nu'/\eta'}(\xx) \widetilde{\mcE}_{\lambda}^{\nu}(\bal, \bbe),
\end{align*}
where, for $\ell' \geq \ell(\nu)$,
\begin{align*}
\mcI_{\lambda}^{\nu}(\bal, \bbe) & = \det \bigl[ h_{\nu_i - \lambda_j - i + j}(A_{\lambda_j} / B_{j-1}) \bigr]_{i,j=1}^{\ell'},
&
\widetilde{\mcI}_{\lambda}^{\nu}(\bal, \bbe) & = \det \bigl[ e_{\nu_i - \lambda_j - i + j}(A_{j-1} / B_{\lambda_j}) \bigr]_{i,j=1}^{\ell'},
\\
\mcD_{\eta}^{\mu}(\bal, \bbe) & =  \det \bigl[ h_{\mu_i - \eta_j - i + j}(B_i / A_{\mu_i}) \bigr]_{i,j=1}^{\ell},
&
\widetilde{\mcD}_{\eta}^{\mu}(\bal, \bbe) & =  \det \bigl[ e_{\mu_i - \eta_j - i + j}(B_{\mu_i} / A_i) \bigr]_{i,j=1}^{\ell'},
\\
\mcE^{\nu}_{\lambda}(\bal, \bbe) & = \det \bigl[ h_{\lambda_i - \nu_j - i + j}(B_{i-1} / A_{\lambda_i-1}) \bigr]_{i,j=1}^{\ell},
&
\widetilde{\mcE}_{\lambda}^{\nu}(\bal, \bbe) & = \det \bigl[ e_{\lambda_i - \nu_j - i + j}(B_{\lambda_i-1} / A_{i-1}) \bigr]_{i,j=1}^{\ell}.
\end{align*}
\end{thm}

\begin{proof}
We begin by showing the first equation.
We have
\[
G_{\lambda/\mu}(\xx; \bal, \bbe) = 
{}^{[[\bal, \bbe]]} \bra{\mu} e^{H(\xx)} \ket{\lambda}^{[\bal,\bbe]}
= \sum_{\eta, \nu} {}^{[[\bal, \bbe]]} \braket{\mu}{\eta} \bra{\eta} e^{H(\xx)} \ket{\nu} \braket{\nu}{\lambda}^{[\bal,\bbe]},
\]
and we see that $\mcD_{\eta}^{\mu} = {}^{[[\bal, \bbe]]} \braket{\mu}{\eta}$ and $\mcI_{\lambda}^{\nu} = \braket{\nu}{\lambda}^{[\bal,\bbe]}$ by Wick's theorem, shown similarly to the proof of Theorem~\ref{thm:dual_basis}.
Recall that $s_{\nu/\eta}(\xx) = \bra{\eta} e^{H(\xx)} \ket{\nu}$, and the claim holds for $\xx = \xx_m$ for any $m$.


For the second equation, the proof is similar to the first except we now need to evaluate
\[
\widetilde{\mcI}_{\lambda}^{\nu} = \braket{\nu'}{\lambda'}^{[\bal,\bbe]} = \mcI_{\lambda'}^{\nu'},
\qquad\qquad
\widetilde{\mcD}_{\eta}^{\mu} = {}^{[[\bal, \bbe]]} \braket{\mu'}{\eta'} = \mcD_{\eta'}^{\mu'},
\]
in terms of the original shapes instead of the conjugate shapes.
We use Lemma~\ref{lemma:conjugate_vectors} and~\eqref{eq:conjugate_form} when applying Wick's theorem as before.

The remaining equations are similar to the first two.
\end{proof}

We remark that the first four expansions in Theorem~\ref{thm:matrix_elements} are infinite, whereas the last two are finite.

We also have $\mcE^{\nu}_{\lambda} = \braket{\nu}{\lambda}_{[\bal,\bbe]}$ as a consequence of~\cite[Ex.~3.3]{Iwao23} using Remark~\ref{rem:multischur}.
Note that this reduces to~\cite[Thm.~3.2, Cor.~3.3]{HJKSS21} when taking $\mu = \emptyset$.
Combinatorial interpretations of the matrix elements $\mcD_{\lambda}^{\mu}(\bal,\bbe)$, $\mcE_{\lambda}^{\mu}(\bal,\bbe)$, and $\mcI_{\lambda}^{\mu}(\bal,\bbe)$ were given in~\cite{HJKSS21}, thus making them (degree signed) Schur positive.
In particular they are polynomials in $\ZZ_{\geq 0}[\pm\bal,\pm\bbe]$ for an appropriate choice of sign of $\pm\bal$ and $\pm\bbe$.

Let us comment on the similarity of Theorem~\ref{thm:matrix_elements} and~\cite[Thm.~8.2, 8.3]{HJKSS21}.
In our formula, if we take $n$ variables (so using $\xx_n$), we can restrict to the case when $\ell(\nu) \leq \ell(\lambda) + n$ as otherwise $s_{\nu/\eta}(\xx_n) = 0$.
However, in~\cite[Thm.~8.2, 8.3]{HJKSS21}, they restrict their all of their partitions to have length at most $m$, but they use ``generalized partitions'' that allow negative parts and the have an extra overall factor that appears in the Jacobi--Trudi formulas.
By using a similar proof to~\cite[Thm.~8.2, 8.3]{HJKSS21}, we can obtain the analogous formulas
\begin{align*}
\G_{\lambda\ds\mu}(\xx_n; \bal, \bbe) & = \prod_{i=1}^{m} \prod_{j=1}^n (1 - \beta_i x_j) \sum_{\overline{\eta} \subseteq \mu \subseteq \lambda \subseteq \nu} (\mcE^*)_{\overline{\eta}}^{\mu}(\bal, \bbe) s_{\nu/\overline{\eta}}(\xx) \mcI_{\lambda}^{\nu}(\bal, \bbe),
\\
\G_{\lambda'\ds\mu'}(\xx_n; \bal, \bbe) & = \prod_{i=1}^{m} \prod_{j=1}^n (1 + \alpha_i x_j)^{-1} \sum_{\overline{\eta} \subseteq \mu \subseteq \lambda \subseteq \nu} (\widetilde{\mcE}^*)_{\eta}^{\mu}(\bal, \bbe) s_{\nu'/\overline{\eta}'}(\xx) \widetilde{\mcI}_{\lambda}^{\nu}(\bal, \bbe),
\end{align*}
where $\overline{\eta}$ is allowed to have negative parts and $\ell(\nu) \leq m$.

We remark that we did not require Theorem~\ref{thm:matrix_elements} to obtain Theorem~\ref{thm:canonical_conjugate} unlike~\cite{HJKSS21}.

\subsection{Branching rules}

We describe branching rules for the (dual) canonical Grothendieck functions, which are formulas where we split the set of variables.
Specifically, we have the following new branching rules, which are refinements of those in~\cite[Prop.~8.7, 8.8]{Yel17}.

\begin{prop}[Branching rules]
\label{prop:branching_rules}
We have
\begin{align*}
\G_{\lambda/\mu}(\xx, \yy; \bal, \bbe) & = \sum_{\nu \subseteq \lambda} \G_{\lambda\ds\nu}(\yy; \bal, \bbe) \G_{\nu/\mu}(\xx; \bal, \bbe),
\\
\G_{\lambda\ds\mu}(\xx, \yy; \bal, \bbe) & = \sum_{\mu \subseteq \nu \subseteq \lambda} \G_{\lambda\ds\nu}(\yy; \bal, \bbe) \G_{\nu\ds\mu}(\xx; \bal, \bbe).
\\
\dG_{\lambda/\mu}(\xx, \yy; \bal, \bbe) & = \sum_{\mu \subseteq \nu \subseteq \lambda} \dG_{\lambda / \nu}(\yy; \bal, \bbe) \dG_{\nu/\mu}(\xx; \bal, \bbe).
\end{align*}
\end{prop}

\begin{proof}
The claim follows by inserting the identity operator
\begin{subequations}
\begin{align}
\id & = \sum_{\lambda} \ket{\lambda}^{[\bal,\bbe]} \cdot {}^{[\bal,\bbe]} \bra{\lambda} \label{eq:identity_upper}
\\ & = \sum_{\lambda} \ket{\lambda}_{[\bal,\bbe]} \cdot {}_{[\bal,\bbe]} \bra{\lambda}\label{eq:identity_lower}
\end{align}
\end{subequations}
in~\eqref{eq:grothendieck_defn}.
\end{proof}

The summation ``$\sum_{\nu \subseteq \lambda}$" cannot be replaced with ``$\sum_{\mu\subseteq \nu \subseteq \lambda}$'' in the first equation.
Furthermore, $\mu \not\subseteq \lambda$ does not imply that $\G_{\lambda/\mu}(\xx; \bal, \bbe) = 0$.
We see that $\G_{\lambda\ds\mu}(\xx; \bal, \bbe) = 0$ whenever $\mu \not\subseteq \lambda$ and $\G_{\lambda \ds \lambda}(\xx; \bal, \bbe) = 1$ by Lemma~\ref{lemma:reduction} and Theorem~\ref{thm:dual_basis}.
However, Equation~\eqref{eq:corner_decomposition} implies that
\[
0 = \G_{(1) \ds (2)}(\xx; \bal, \bbe) = \G_{(1) / (2)}(\xx; \bal, \bbe) - (\alpha_2 + \beta_1) \G_{(1) / (1)}(\xx; \bal, \bbe).
\]
We can see from the Jacobi--Trudi formula that $\G_{(1)/(1)}(\xx; \bal, \bbe) = 1$, and so
\[
\G_{(1) / (2)}(\xx; \bal, \bbe) = \alpha_2 + \beta_1.
\]
For general $\mu \not\subseteq \lambda$, it is straightforward to derive 
\begin{equation}\label{eq:single_slash_ex}
G_{\lambda/\mu}(\xx; \bal, \bbe)=\prod_{(i,j)\in \mu/\lambda}(\alpha_i+\beta_j)\cdot G_{\lambda/(\lambda\cap \mu)}(\xx; \bal, \bbe)
\end{equation}
from Equation~\eqref{eq:inverse_corner} and $G_{\lambda\ds\mu}(\xx; \bal, \bbe)=0$ $(\mu \not\subseteq \lambda)$.
One may use \eqref{eq:single_slash_ex} as an alternative definition of $G_{\lambda/\mu}(\xx; \bal, \bbe)$ ($\mu \not\subseteq \lambda$)
extending~\cite[Def.~6.5]{HJKSS21} (which is assuming $\mu \subseteq \lambda$).
The non-vanishing of $\G_{\lambda/\mu}(\xx; \bal, \bbe)$ when $\mu \not\subseteq \lambda$ is reflecting the fact that the skew shape description for $G_{\lambda/\mu}(\xx; \bal, \bbe)$ is not natural from the point of view of branching rules.

\subsection{Skew Cauchy identities}

Our next identities are the skew Cauchy identities, which are the refined canonical version of~\cite[Thm.~5.1]{Yel19}
and~\cite[Cor.~6.3]{Yel19}.

\begin{thm}[Skew Cauchy formulas]
\label{thm:skew_cauchy}
We have
\begin{subequations}
\begin{align}
\label{eq:skew_cauchy_basic}
\sum_{\lambda} \G_{\lambda \ds \mu}(\xx; \bal, \bbe) \dG_{\lambda / \nu}(\yy; \bal, \bbe) 
&= \prod_{i,j} \frac{1}{1 - x_i y_j} \sum_{\eta}  \G_{\nu \ds \eta}(\xx; \bal, \bbe)\dG_{\mu / \eta}(\yy; \bal, \bbe), \\
\sum_{\lambda} \G_{\lambda^\prime \ds \mu^\prime}(\xx; \bbe, \bal) \dG_{\lambda / \nu}(\yy; \bal, \bbe) 
&= \prod_{i,j} (1 + x_i y_j) \sum_{\eta}  \G_{\nu^\prime \ds \eta^\prime}(\xx; \bbe, \bal)\dG_{\mu / \eta}(\yy; \bal, \bbe), \\
\sum_{\lambda} \G_{\lambda \ds \mu}(\xx; \bal, \bbe) \dG_{\lambda^\prime / \nu^\prime}(\yy; \bbe, \bal) 
&= \prod_{i,j} (1 + x_i y_j) \sum_{\eta}  \G_{\nu \ds \eta}(\xx; \bal, \bbe)\dG_{\mu^\prime / \eta^\prime}(\yy; \bbe, \bal), \\
\sum_{\lambda} \G_{\lambda^\prime \ds \mu^\prime}(\xx; \bbe, \bal) \dG_{\lambda^\prime / \nu^\prime}(\yy; \bbe, \bal) 
&= \prod_{i,j} \frac{1}{1 - x_i y_j} \sum_{\eta}  \G_{\nu^\prime \ds \eta^\prime}(\xx; \bbe, \bal)\dG_{\mu^\prime / \eta^\prime}(\yy; \bbe, \bal).
\end{align}
\end{subequations}
\end{thm}

\begin{proof}
We show Equation~\eqref{eq:skew_cauchy_basic} as the remaining identities can be derived by $\omega$-involution on $\xx$ and/or $\yy$ or can be proved in the same manner.
Equation~\eqref{eq:skew_cauchy_basic} follows by evaluating
\begin{equation}
\label{eq:evaluateforskewcauchy}
{}^{[\bal,\bbe]} \bra{\mu} e^{H(\xx)} e^{H^*(\yy)} \ket{\nu}^{[\bal,\bbe]}
\end{equation}
in two ways using the identity operator~\eqref{eq:identity_upper}.

First~\eqref{eq:evaluateforskewcauchy} can be evaluated using Corollary~\ref{cor:free_fermion_grothendiecks} with $\ast$ being applied to the dual canonical formula:
\begin{equation}
\label{eq:skew_cauchy1}
\begin{aligned}
{}^{[\bal,\bbe]} \bra{\mu} e^{H(\xx)} e^{H^*(\yy)} \ket{\nu}^{[\bal,\bbe]} & = \sum_{\lambda} {}^{[\bal,\bbe]} \bra{\mu} e^{H(\xx)} \ket{\lambda}^{[\bal,\bbe]} \cdot {}^{[\bal,\bbe]} \bra{\lambda} e^{H^*(\yy)} \ket{\nu}^{[\bal,\bbe]}
\\
& = \sum_{\lambda} \G_{\lambda\ds\mu}(\xx; \bal, \bbe) \dG_{\lambda/\nu}(\yy; \bal, \bbe).
\end{aligned}
\end{equation}
Another way to evaluate is to first use
\[
e^{H(\xx)}e^{H^*(\yy)} = \prod_{i,j}\frac{1}{1-x_i y_j} e^{H^*(\yy)}e^{H(\xx)},
\]
which follows from~\eqref{eq:boson_relation}, to compute
\begin{align*}
{}^{[\bal,\bbe]} \bra{\mu} e^{H(\xx)} e^{H^*(\yy)} \ket{\nu}^{[\bal,\bbe]} & = \prod_{i,j}\frac{1}{1 - x_i y_j} {}^{[\bal,\bbe]} \bra{\mu} e^{H^*(\yy)} e^{H(\xx)} \ket{\nu}^{[\bal,\bbe]}
\\ & = \prod_{i,j}\frac{1}{1 - x_i y_j} \sum_{\eta} {}^{[\bal,\bbe]} \bra{\mu} e^{H^*(\yy)} \ket{\eta}^{[\bal,\bbe]} \cdot {}^{[\bal,\bbe]} \bra{\eta} e^{H(\xx)} \ket{\nu}^{[\bal,\bbe]}
\\ & = \prod_{i,j} \frac{1}{1 - x_i y_j} \sum_{\eta} \dG_{\mu / \eta}(\yy; \bal, \bbe) \G_{\nu \ds \eta}(\xx; \bal, \bbe).
\end{align*}
Comparing this with~\eqref{eq:skew_cauchy1} yields the result.
\end{proof}

\subsection{Skew Pieri-type identities}

We can also derive the refined canonical versions of skew Pieri-type formulas~\cite[Thm.~7.10]{Yel19} using free-fermions
(the skew Pieri-type formulas can also be derived from the skew-Cauchy formulas by
 Warnaar's general argument~\cite{Warnaar13} as explained in ~\cite[Section 7]{Yel19}).

\begin{thm}[Skew Pieri-type formulas]
\label{thm: pieri-formulas}
\begin{subequations}
\begin{align}
\sum_{\lambda,\eta} G_{\lambda \ds \mu}({\bf x};\bal,\bbe) G_{\nu^\prime \ds \eta^\prime}(-{\bf x};\bbe,\bal)
g_{\lambda/\eta}({\bf y};\bal,\bbe)&=\prod_{i,j} \frac{1}{1-x_i y_j}
g_{\mu/\nu}({\bf y};\bal,\bbe), \label{skewpierione} \\
\sum_{\lambda,\eta} g_{\lambda / \mu}({\bf x};\bal,\bbe) g_{\nu^\prime / \eta^\prime}(-{\bf x};\bbe,\bal)
G_{\lambda \ds \eta}({\bf y};\bal,\bbe)&=\prod_{i,j} \frac{1}{1-x_i y_j}
G_{\mu \ds \nu}({\bf y};\bal,\bbe), \label{skewpieritwo} \\
\sum_{\lambda,\eta} G_{\lambda^\prime \ds \mu^\prime}({\bf x};\bbe,\bal) G_{\nu \ds \eta}(-{\bf x};\bal,\bbe)
g_{\lambda/\eta}({\bf y};\bal,\bbe)&=\prod_{i,j} (1+x_i y_j)
g_{\mu/\nu}({\bf y};\bal,\bbe), \label{skewpierithree} \\
\sum_{\lambda,\eta} g_{\lambda^\prime / \mu^\prime}({\bf x};\bbe,\bal) g_{\nu / \eta}(-{\bf x};\bal,\bbe)
G_{\lambda \ds \eta}({\bf y};\bal,\bbe)&=\prod_{i,j} (1+x_i y_j)
G_{\mu \ds \nu}({\bf y};\bal,\bbe). \label{skewpierifour}
\end{align}
\end{subequations}
\end{thm}

\begin{proof}
\eqref{skewpierione}
can be derived by
evaluating
\begin{eqnarray*}
{}^{[\bal,\bbe]} \bra{\mu}e^{H({\bf x})} e^{H^*(\bf y)} e^{-H(\bf x)} \ket{\nu}^{[\bal,\bbe]},
\end{eqnarray*}
in two ways. One way of evaluation is
\begin{equation}
\label{skewpierione-one}
\begin{aligned}
&{}^{[\bal,\bbe]} \bra{\mu}e^{H({\bf x})} e^{H^*(\bf y)} e^{-H(\bf x)} \ket{\nu}^{[\bal,\bbe]} \\
& \hspace{40pt} = \prod_{i,j} \frac{1}{1-x_i y_j} {}^{[\bal,\bbe]} \bra{\mu} e^{H^*(\bf y)} \ket{\nu}^{[\bal,\bbe]} 
=\prod_{i,j} \frac{1}{1-x_i y_j} g_{\mu/\nu}({\bf y};\bal,\bbe).
\end{aligned}
\end{equation}
Another way of evaluation is to insert decomposition of the identity operator as
\begin{align}
&{}^{[\bal,\bbe]} \bra{\mu}e^{H({\bf x})} e^{H^*(\bf y)} e^{-H(\bf x)} \ket{\nu}^{[\bal,\bbe]} \nonumber \\
& \hspace{40pt} = \sum_{\lambda,\eta} {}^{[\bal,\bbe]} \bra{\mu}e^{H(\bf x)} \ket{\lambda}^{[\bal,\bbe]} \cdot
{}^{[\bal,\bbe]} \langle \lambda| e^{H^*(\bf y)} \ket{\eta}^{[\bal,\bbe]} \cdot
{}^{[\bal,\bbe]} \langle \eta|e^{-H(\bf x)} \ket{\nu}^{[\bal,\bbe]} \nonumber \\
& \hspace{40pt} = \sum_{\lambda,\eta} {}^{[\bal,\bbe]} \bra{\mu}e^{H(\bf x)} \ket{\lambda}^{[\bal,\bbe]} \cdot
{}^{[\bal,\bbe]} \langle \lambda| e^{H^*(\bf y)} \ket{\eta}^{[\bal,\bbe]} \cdot
{}^{[\bal,\bbe]} \langle \eta|e^{\omega H(-\bf x)} \ket{\nu}^{[\bal,\bbe]} \nonumber \\
& \hspace{40pt} = \sum_{\lambda,\eta} G_{\lambda \ds \mu}({\bf x};\bal,\bbe) 
g_{\lambda/\eta}({\bf y};\bal,\bbe) G_{\nu^\prime \ds \eta^\prime}(-{\bf x};\bbe,\bal).
\label{skewpierione-two}
\end{align}
Equation~\eqref{skewpierione} follows from \eqref{skewpierione-one} and \eqref{skewpierione-two}.

We can also derive \eqref{skewpieritwo} in the same way by evaluating
${}_{[\bal,\bbe]} \bra{\mu}e^{H(\xx)} e^{H^*(\yy)} e^{-H(\xx)} \ket{\nu}_{[\bal,\bbe]}$.
\eqref{skewpierithree} and \eqref{skewpierifour} can be derived in the same way or
follows from \eqref{skewpierione} and \eqref{skewpieritwo} by $\omega$-involution on $\xx$ respectively.
\end{proof}

\subsection{Determinantal formulas}

We show an additional determinantal formulas for $G_{\lambda\ds\mu}(\xx;\bal, \bbe)$ and $G_{\lambda/\mu}(\xx;\bal, \bbe)$ by using their fermionic presentation. 
From~\cite[Rem.~2.8]{HIMN17}, let $\mathfrak{G}_k(\xx,\beta)$ be the single row $\beta$-Grothendieck polynomial, which is the coefficient of $z^k$ in the generating function
\[
\mathcal{G}(z,\beta) := \sum_{k \in \ZZ} \mathfrak{G}_k(\xx,\beta) z^k = \frac{1}{1-\beta z^{-1}} \prod_{k=1}^\infty \frac{1 - \beta x_k}{1 - x_k z}.
\]

\begin{remark}
\label{rem:expansions}
We must expand $(1 - \beta z^{-1})^{-1} = 1 + \beta z^{-1} + \cdots \in \ZZ[\beta][\![z^{-1}]\!]$ and be careful about the cancellations rather than rewriting it as the formal power series $\frac{z}{z - \beta} = -\beta^{-1}z - \beta^{-2}z^2 + \cdots$ as $\beta$ may not be invertible.
If we instead multiply by $\frac{z}{z - \beta} \in \ZZ[\beta^{-1}][\![z]\!]$, then we have an extra term of $-\beta^{-k}$ in $z^k$ in the expansion.
\end{remark}

The following is proved in \cite[Cor.~3.4]{Iwao20}:
\begin{align}\label{eq:generate_G}
e^{H(\xx)} e^{H^*(\beta)} \psi(z) e^{-H^*(\beta)} e^{-H(\xx)}
=\prod_{k=1}^\infty (1-\beta x_k)^{-1} \mathcal{G}(z,\beta) \psi(z).
\end{align}

\begin{prop}\label{prop:determinant_formula_for_Gs}
If $\ell>\ell(\lambda)$, we have
\begin{align*}
G_{\lambda / \mu}(\xx;\bal, \bbe) 
&=\det \Bigg[ \sum_{m=0}^\infty \sum_{k=0}^{\infty} 
(-1)^k e_k(A_{\mu_j} \sqcup B_{(j,\ell]} )
h_{m-k}(A_{\lambda_i}\sqcup B_{(i,\ell]}   ) 
\mathfrak{G}_{\lambda_i-\mu_j-i+j+m}(\xx,\beta_i) \Bigg]_{i,j=1}^{\ell}, \\
G_{\lambda \ds \mu}(\xx;\bal, \bbe) 
&= \det \Bigg[ \sum_{m=0}^\infty \sum_{k=0}^{\mu_j-j+\ell} 
(-1)^k e_k(A_{\mu_j-1} \sqcup B_{[j,\ell]} )
 h_{m-k}(A_{\lambda_i} \sqcup B_{(i,\ell]} ) 
\mathfrak{G}_{\lambda_i-\mu_j-i+j+m}(\xx,\beta_i) \Bigg]_{i,j=1}^{\ell}.
\end{align*}
\end{prop}

\begin{proof}
This can be proved in the same way as discussed in \cite[Sec.~4.1]{Iwao20}.
Let us show the case for $G_{\lambda \ds \mu}(\xx;\bal, \bbe)$.
The case for $G_{\lambda / \mu}(\xx;\bal, \bbe)$ can be proven in a similar way.

We introduce the following generating function for $G_{\lambda \ds \mu}(\xx;\bal, \bbe)$
\begin{align}
\Phi:=\Phi(z_1,\dotsc,z_\ell,w_1,\dotsc,w_\ell)
=& \bra{-\ell} \prod^{\leftarrow}_{1 \leq j \leq \ell}
\left( e^{-H^\ast(B_{j-1}/A_{\mu_j-1})} \psi^*(w_j)  e^{H^\ast(B_{j-1}/A_{\mu_j-1})}  \right) 
e^{H(\xx)} \nonumber \\
&\times 
e^{-H^\ast(B_\ell)}
\prod^{\rightarrow}_{1 \leq i \leq \ell} 
\left( 
e^{H^*(A_{\lambda_i}\sqcup B_{[i,\ell]})} 
\psi(z_i)
e^{-H^*(A_{\lambda_i}\sqcup B_{[i,\ell]})} 
 \right) 
\ket{-\ell},
\label{refinedskewgrothendieckgenerating}
\end{align}
whose coefficient of $\prod_{i=1}^\ell z_i^{\lambda_i-i} \prod_{j=1}^\ell w_j^{\mu_j-j}$
is $G_{\lambda \ds \mu}(\xx;\bal, \bbe)$.
We first rewrite $\Phi$
as
\begin{align}
\Phi=\bra{-\ell} Q_\ell(w_\ell) \cdots Q_2(w_2) Q_1(w_1) e^{H^*(B_{\ell})} e^{H(\xx)}
e^{-H^*(B_{\ell})} P_1(z_1) P_2(z_2) \cdots P_\ell(z_\ell) \ket{-\ell},
\end{align}
where
\begin{align*}
P_i(z_i)  = e^{H^*(A_{\lambda_i}\sqcup B_{[i,\ell]})} \psi(z_i) e^{-H^*(A_{\lambda_i}\sqcup B_{[i,\ell]})},
\quad
Q_j(w_j) = 
e^{H^*(A_{\mu_j-1}\sqcup B_{[j,\ell]})}
\psi^*(w_j) 
e^{-H^*(A_{\mu_j-1}\sqcup B_{[j,\ell]})}.
\end{align*}
Next, we use~\eqref{eq:conjugate_op_star}, $e^{-H(\xx)} \ket{-\ell} = \ket{-\ell}$ and Wick's theorem to get the determinant form
\begin{align}
\Phi
&=
\prod_{j=1}^\ell \prod_{k=1}^\infty
(1-\beta_j x_k)
\times
\bra{-\ell} Q_\ell(w_\ell) \cdots Q_1(w_1) 
(e^{H(\xx)} P_1(z_1) e^{-H(\xx)})
\cdots
(e^{H(\xx)} P_\ell(z_\ell) e^{-H(\xx)})
\ket{-\ell}, \nonumber \\
& = \prod_{j=1}^\ell \prod_{k=1}^\infty
(1-\beta_j x_k)\cdot
\det \bigl[
\bra{-\ell} Q_j(w_j) e^{H(\xx)} P_i(z_i) e^{-H(\xx)} \ket{-\ell}
\bigr]_{i,j=1}^\ell.
\label{skewGrothendieckgeneratingdetformtochu}
\end{align}
Substituting 
\[
\begin{aligned}
e^{H(\xx)}P_i(z_i)e^{-H(\xx)}
&=
e^{H^\ast(A_{\lambda_i} \sqcup B_{(i,\ell]})}
\left(e^{H^\ast(\beta_i)}e^{H(\xx)}
\psi(z_i)
e^{-H(\xx)}e^{-H^\ast(\beta_i)}\right)
e^{-H^\ast(A_{\lambda_i} \sqcup B_{(i,\ell]})}\\
&=
\left(\sum_{m=0}^\infty h_{m}(A_{\lambda_i}\sqcup B_{(i,\ell]})z_i^{-m}\right)
\prod_{k=1}^\infty (1-\beta_i x_k)^{-1} \mathcal{G}(z_i,\beta_i) \psi(z_i),
\label{Cgauge}
\end{aligned}
\]
\[
\begin{aligned}
Q_j(w_j)& = e^{H^*(A_{\mu_j-1}\sqcup B_{[j,\ell]})}
\psi^*(w_j) 
e^{-H^*(A_{\mu_j-1}\sqcup B_{[j,\ell]})}
=\left(\sum_{k=0}^\infty (-1)^ke_{k}(A_{\mu_j-1}\sqcup B_{[j,\ell]})w_j^k\right) \psi^*(w_j), 
\end{aligned}
\]
and 
\[
\bra{-\ell} \psi^*(w_j) \psi(z_i) \ket{-\ell} = \sum_{p=-\ell}^\infty z_i^p w_j^p, 
\]
to~\eqref{skewGrothendieckgeneratingdetformtochu}, we obtain
\begin{align}
\Phi
=\det \Bigg[
\left(\sum_{k=0}^\infty (-1)^ke_{k}(A_{\mu_j-1}\sqcup B_{[j,\ell]})w_j^k\right)
\left(\sum_{m=0}^\infty h_{m}(A_{\lambda_i}\sqcup B_{(i,\ell]})z_i^{-m}\right)
\mathcal{G}(z_i,\beta_i)
\sum_{p=-\ell}^\infty z_i^p w_j^p
\Bigg]_{i,j=1}^\ell. 
\label{skewGrothendieckgeneratingdetform}
\end{align}
Extracting the coefficients of $\prod_{i=1}^\ell z_i^{\lambda_i-i} \prod_{j=1}^\ell w_j^{\mu_j-j}$ of both hand sides of~\eqref{skewGrothendieckgeneratingdetform} gives the claimed determinant representation for $G_{\lambda \ds \mu}(\xx;\bal, \bbe)$.
\end{proof}

\begin{remark}
The summation ``$\sum_{k=0}^{\mu_j-j+\ell}$'' in the determinant formula for $G_{\lambda\ds\mu}(\xx;\bal,\bbe)$ (Proposition \ref{prop:determinant_formula_for_Gs}) cannot be replaced with ``$\sum_{k=0}^\infty$'' because $e_{\mu_j-j+\ell+1}(A_{\mu_j-1} \sqcup B_{[j,\ell]} )$ is not $0$ if $\mu_j=0$.
Compare to $e_{\mu_j-j+\ell+1}(A_{\mu_j} \sqcup B_{(j,\ell]} )=0$.
\end{remark}

Proposition~\ref{prop:determinant_formula_for_Gs} roughly differs from the Jacobi--Trudi formulas by replacing $h_k(\xx)$ with $\mathfrak{G}_{k}(\xx; \beta)$.
This yields a distinct formula since $\mathfrak{G}_k(\xx; \beta) = \sum_{a=0}^{\infty} \beta^a s_{k1^a}(\xx)$ is a sum over Schur functions with a hook shape (see, \textit{e.g.},~\cite{Lenart00}).

\subsection{Decomposition of one set of parameters}
\label{sec:one_parameter_decomposition}

We can also answer~\cite[Prob.~12.2]{Yel17} by using
\begin{align*}
\dG_{\lambda}(\xx; \bal, \bbe) & = \bra{0} e^{H(\xx)} \ket{\lambda}_{[\bal,\bbe]}
= \sum_{\mu} \bra{0} e^{H(\xx)} \ket{\mu}_{[\bbe]} \cdot {}_{[\bbe]}\braket{\mu}{\lambda}_{[\bal,\bbe]}
\\ & = \sum_{\mu} b_{\lambda}^{\mu} \bra{0} e^{H(\xx)} \ket{\mu}_{[\bbe]}
= \sum_{\mu \subseteq \lambda} b_{\lambda}^{\mu} \dG_{\mu}(\xx; \bbe),
\end{align*}
where $\dG_{\mu}(\xx; \bbe) := \dG_{\mu}(\xx; 0, \bbe)$.

\begin{cor}
We have $b_{\lambda}^{\mu} \in \ZZ_{\geq 0}[\bal, \bbe]$.
\end{cor}

\begin{proof}
Comparing~\eqref{eq:dualb_det_conjugate} and the flagged Jacobi--Trudi formula from~\cite[Thm~1.7]{HJKSS21} and the corresponding known positivity formula yields the result.
\end{proof}

\begin{remark}
\label{rem:combinatorai_description_b}
From the combinatorial description in~\cite{HJKSS21}, we can interpret $b_{\lambda}^{\mu}$ as the generating function of all flagged valued-set tableau of (skew) shape $\lambda'/\mu'$ such that the entries in row $i$ is strictly less than~$i$
\end{remark}

\begin{ex}
We have
\[
\dG_{222}(\xx; \bal, \bbe) = \dG_{222}(\xx; \bbe) + \alpha_1 \dG_{221}(\xx; \bbe) + \alpha_1 (\alpha_1 + \beta_2) \dG_{211}(\xx; \bbe) + \alpha_1 (\alpha_1 + \beta_1) (\alpha_1 + \beta_2) \dG_{111}(\xx; \bbe).
\]
In particular, we note that the coefficient of $\dG_{111}(\xx; \bbe)$ when specialized to $\bal = \alpha$ and $\bbe = \beta$ is
\[
\alpha^3 + 2 \alpha^2 \beta + \alpha \beta^2,
\]
which shows that~\cite[Prob.~12.2(b)]{Yel17} is answered in the negative.
\end{ex}

\begin{ex}
We can also answer~\cite[Prob.~12.2(c)]{Yel17} in the negative.
Let $\lambda = 54321$ and $\mu = 51111$, which are both self-conjugate partitions.
Taking the specialization $\bal = \alpha$ and $\bbe = \beta$, we have
\[
b_{\lambda}^{\mu} = b_{\lambda'}^{\mu'} = \alpha^6 + 3 \alpha^5 \beta + 3 \alpha^4\beta^2 + \alpha^3 \beta^3,
\]
which has a term of the form $(\alpha \beta)^k$ (a ``free term'' as defined in~\cite[Sec.~12]{Yel17}).
This is still answered in the negative even if both partitions are not self-conjugate.
Consider $\lambda = 55533 = \lambda'$ and $\mu = 53111$, then we have
\begin{align*}
\beta_{\lambda}^{\mu} & = 6 \alpha^{10} + 28 \alpha^9 \beta + 53 \alpha^8 \beta^2 + 52 \alpha^7 \beta^3 + 28 \alpha^6 \beta^4 + 8 \alpha^5 \beta^5 + \alpha^4 \beta^6,
\\
\beta_{\lambda'}^{\mu'} & = 10 \alpha^{10} + 44 \alpha^9 \beta + 79 \alpha^8 \beta^2 + 74 \alpha^7 \beta^3 + 38 \alpha^6 \beta^4 + 10 \alpha^5 \beta^5 + \alpha^4 \beta^6.
\end{align*}
\end{ex}

Additionally, we have $\G_{\lambda}(\xx; \bal, \bbe) = \sum_{\lambda \subseteq \mu} B_{\lambda}^{\mu} \G_{\mu}(\xx; \bbe)$, where $\G_{\mu}(\xx; \bbe) := \G_{\mu}(\xx; 0, \bbe)$.

\begin{cor}
\label{cor:b_det_combinatorial}
We have
$
B_{\lambda}^{\mu} \in \ZZ_{\geq 0}[-\bal, -\bbe].
$
\end{cor}

\begin{proof}
Comparing~\eqref{eq:b_det} and the flagged Jacobi--Trudi formula from~\cite[Thm~1.7]{HJKSS21} after taking the transpose $i \leftrightarrow j$ and $X_{[m-\lambda_i+1,m]}$ (in the notation of~\cite{HJKSS21}), for some fixed $m$ with $x_{m-i} = \alpha_i$, yields the equality.
\end{proof}

\begin{remark}
Similar to Remark~\ref{rem:combinatorai_description_b}, we can describe $B_{\lambda}^{\mu}$ as the generating function of flagged plane partitions with the entries in row $i$ being strictly less than $\lambda_i$.
\end{remark}

A natural question would be to compare these flagged plane partitions with the recording tableau from the uncrowding algorithm of~\cite[Def.~3.5]{PPPS20}.
However, we must be careful about the signs, because while the signs in our formulas only depend on the degree, they appear in opposite ways between the combinatorial definition of $\G_{\lambda}(\xx; \alpha, \beta)$ and the expansion coefficients $B_{\lambda}^{\mu}$.
Therefore, it is not possible to make a combinatorial comparison between the two constructions.
Compare this with~\cite[Cor.~4.9]{PPPS20}, which is very close to our expansion formula from Corollary~\ref{cor:b_det_combinatorial}.

\begin{ex}
We compute
\begin{align*}
G_{11}(\xx; \bal, \bbe) & = \G_{11}(\xx; \bbe) + \alpha_1 \G_{21}(\xx; \bbe) + \alpha_1 \beta_1 \G_{22}(\xx; \bbe) + \cdots
\\ & = s_{11} - \alpha_1 s_{21} - (\beta_1 + \beta_2) s_{111}
\\ & \hspace{20pt} + \alpha_1^2 s_{31} + (\alpha_1\beta_1 + \alpha_1\beta_2) s_{2111} + (\beta_1^2 + \beta_1 \beta_2 + \beta_2^2) s_{1111} + \cdots.
\end{align*}
In particular, the coefficient of $s_{22}(\xx)$ is $0$, which comes from the cancellation of the term in
\[
\G_{21}(\xx; \bbe) = s_{21} - \beta_1 s_{22} + \cdots,
\qquad\qquad
\G_{22}(\xx; \bbe) = s_{22} + \cdots.
\]
\end{ex}

We could additionally compute similar expansion formulas for $\G_{\lambda}(\xx; \bal, \bbe)$ into $\wG_{\mu}(\xx; \bal)$ and for $\dG_{\lambda}(\xx; \bal, \bbe)$ into $\dwG_{\mu}(\xx; \bal)$ with similar tableau formulas.
We leave the details for the interested reader.

\subsection{Integral formulas}

The Jacobi--Trudi formulas given in Theorem~\ref{thm:jacobi_trudi} can be converted to integral formulas.

\begin{thm}
\label{thm:integral_G}
For $\ell(\lambda) \le \ell$, we have
\begin{subequations}
\begin{align}
&g_{\lambda/\mu}(\xx_{n};\bal,\bbe)
\nonumber \\
& \hspace{15pt} =\det\Bigg[
\frac{1}{2 \pi i}
\oint_{\gamma_r}
\frac{
\prod_{k=1}^{\lambda_i-1}(1+\alpha_k w) \prod_{k=1}^{j-1} (1-\beta_k w)
}{
\prod_{k=1}^{\mu_j}(1+\alpha_k w) \prod_{k=1}^{i-1} (1-\beta_k w)
\prod_{m=1}^n (1-x_m w) w^{\lambda_i-\mu_j-i+j+1}} dw
\Bigg]_{i,j=1}^\ell, \label{eq:dualintegral}
\\
&G_{\lambda/\mu}(\xx_{n};\bal,\bbe)=
\prod_{i=1}^\ell \prod_{j=1}^n (1-\beta_i x_j)
\nonumber \\
& \hspace{15pt} \times\det\Bigg[
\frac{1}{2 \pi i}
\oint_{\widetilde{\gamma}_r}
\frac{
\prod_{k=1}^{{\mu_j}
}(1+\alpha_k w^{-1}) \prod_{k=1}^{i-1} (1-\beta_k w^{-1})
}{
\prod_{k=1}^{\lambda_i}(1+\alpha_k w^{-1}) \prod_{k=1}^{j} (1-\beta_k w^{-1})
\prod_{m=1}^n (1-x_m w) w^{\lambda_i-\mu_j-i+j+1}} dw
\Bigg]_{i,j=1}^\ell, \label{integralGone} \\
&G_{\lambda\ds\mu}(\xx_{n};\bal,\bbe)=
\prod_{i=1}^\ell \prod_{j=1}^n (1-\beta_i x_j)
\nonumber \\
& \hspace{15pt}\times\det\Bigg[
\frac{1}{2 \pi i}
\oint_{\widetilde{\gamma}_r}
\frac{
\prod_{k=1}^{\mu_j-1}(1+\alpha_k w^{-1}) \prod_{k=1}^{i-1} (1-\beta_k w^{-1})
}{
\prod_{k=1}^{\lambda_i}(1+\alpha_k w^{-1}) \prod_{k=1}^{j-1} (1-\beta_k w^{-1})
\prod_{m=1}^n (1-x_m w) w^{\lambda_i-\mu_j-i+j+1}} dw
\Bigg]_{i,j=1}^\ell, \label{integralGtwo}
\end{align}
\end{subequations}
where $\gamma_r$ is a circle centered at the origin with radius $r$ satisfying $0 < r < \abs{x_m^{-1}}$ for $m=1,\dotsc,n$ and $0 < r < \abs{\alpha_1^{-1}}, \abs{\alpha_2^{-1}},\dotsc,\abs{\beta_1^{-1}}, \abs{\beta_2^{-1}},\ldots$ and $\widetilde{\gamma}_r$ is a circle centered at the origin with radius $r$ satisfying $0 < r < \abs{x_m^{-1}}$ for $m = 1, \dotsc, n$ and $r > \abs{\alpha_1}, \abs{\alpha_2}, \dotsc,\abs{\beta_1}, \abs{\beta_2}, \ldots$.
\end{thm}

\begin{proof}
We prove~\eqref{eq:dualintegral} as the other identities are proved similarly.

Let us  express $h_{\lambda_i-\mu_j-i+j}(\xx_{n} \sqcup A_{{[} \lambda_i,\mu_j {]}} \sqcup B_{{[}j,i {)}})$ in
\begin{equation*}
g_{\lambda/\mu}(\xx_{n};\bal,\bbe) = \det \bigl[ h_{\lambda_i-\mu_j-i+j}(\xx_{n} \sqcup A_{[ \lambda_i,\mu_j]} \sqcup B_{[j,i)}) \bigr]_{i,j=1}^\ell,
\end{equation*}
derived in Theorem \ref{thm:jacobi_trudi}
as integrals.
Using $h_k(\xx_{n})=0$ for $k<0$ and
\begin{equation}
h_k(\xx_{n})=\frac{1}{2 \pi i} \oint_{\gamma_r}
\frac{1}{\prod_{m=1}^n (1-x_m w) w^{k+1}} dw, \label{integralformhk}
\end{equation}
 where $\gamma_r$ is a circle centered at the origin with radius $r$
satisfying $0 < r < |x_m^{-1}|$, $m=1,\dots,n$,
we have
\begin{align}
h_{\lambda_i-\mu_j-i+j}(\xx_{n}
\sqcup A_{{[} \lambda_i,\mu_j {]}}
\sqcup B_{{[}j,i {)}}) 
&=\sum_{k=0}^{\infty}
h_{\lambda_i-\mu_j-i+j-k}(\xx_{n})
h_k(
A_{{[} \lambda_i,\mu_j {]}}
\sqcup B_{{[}j,i {)}}
) \nonumber \\
&=\sum_{k=0}^{\infty} \frac{1}{2 \pi i}
\oint_{\gamma_r}
\frac{h_k(
A_{{[} \lambda_i,\mu_j {]}}
\sqcup B_{{[}j,i {)}}
)}{\prod_{m=1}^n (1-x_m w) w^{\lambda_i-\mu_j-i+j-k+1}} dw
\nonumber \\
&=\frac{1}{2 \pi i}
\oint_{\gamma_r}
\frac{ \sum_{k=0}^{\infty} h_k(
A_{{[} \lambda_i,\mu_j {]}}
\sqcup B_{{[}j,i {)}}
) w^k
}{\prod_{m=1}^n (1-x_m w) w^{\lambda_i-\mu_j-i+j+1}} dw. \nonumber
\end{align}
One can also show the sum in the integrand above can be rewritten as
\begin{align}
\displaystyle
\sum_{k=0}^{\infty} h_k(
A_{{[} \lambda_i,\mu_j {]}}
\sqcup B_{{[}j,i {)}}
) w^k
=\frac{\prod_{k=1}^{\lambda_i-1}(1+\alpha_k w) \prod_{k=1}^{j-1} (1-\beta_k w) }{\prod_{k=1}^{\mu_j}(1+\alpha_k w) \prod_{k=1}^{i-1} (1-\beta_k w)}. \nonumber
\end{align}
From the computations above, we have
\begin{align}
&h_{\lambda_i-\mu_j-i+j}(\xx_{n}
\sqcup A_{{[} \lambda_i,\mu_j {]}}
\sqcup B_{{[}j,i {)}}) \nonumber \\
&=\frac{1}{2 \pi i}
\oint_{\gamma_r}
\frac{
\prod_{k=1}^{\lambda_i-1}(1+\alpha_k w) \prod_{k=1}^{j-1} (1-\beta_k w)
}{
\prod_{k=1}^{\mu_j}(1+\alpha_k w) \prod_{k=1}^{i-1} (1-\beta_k w)
\prod_{m=1}^n (1-x_m w) w^{\lambda_i-\mu_j-i+j+1}} dw, \nonumber
\end{align}
where we further impose
$r < \abs{\alpha_1^{-1}}, \abs{\alpha_2^{-1}}, \dotsc, \abs{\beta_1^{-1}}, \abs{\beta_2^{-1}}, \ldots$,
so that the origin is the only pole surrounded by the integration contour.
Hence we get~\eqref{eq:dualintegral}.
\end{proof}

The other integral formulas below are proven similarly except using the dual Jacobi--Trudi formulas and identities such as
\[
\sum_{k \ge 0} e_k(A_{{[} i,j {)}}
\sqcup B_{{(}\mu_j^\prime,\lambda_i^\prime {)}}
) w^k
=\frac{
\prod_{k=1}^{j-1}(1-\alpha_k w) \prod_{k=1}^{\lambda_i^\prime-1} (1+\beta_k w)
}{
\prod_{k=1}^{i-1}(1-\alpha_k w) \prod_{k=1}^{\mu_j^\prime} (1+\beta_k w)}.
\]

\begin{thm}
\label{thm:integral_Gconj}
For $\ell^\prime \ge \lambda_1$,
we have
\begin{subequations}
\begin{align}
g_{\lambda/\mu}({\bf x}_n;\bal, \bbe) & = \det\Bigg[
\frac{1}{2 \pi i}
\oint_{\gamma_r}
\frac{
\prod_{k=1}^{j-1}(1-\alpha_k w) \prod_{k=1}^{\lambda_i^\prime-1} (1+\beta_k w)
\prod_{m=1}^n (1+x_m w)
}{
\prod_{k=1}^{i-1}(1-\alpha_k w) \prod_{k=1}^{\mu_j^\prime} (1+\beta_k w)
w^{\lambda_i^\prime-\mu_j^\prime-i+j+1}} dw
\Bigg]_{i,j=1}^{\ell^\prime}, \label{gdualintegral}
\\
G_{\lambda/\mu}(\xx_{n};\bal,\bbe) & =
\prod_{i=1}^{\ell^\prime} \prod_{j=1}^n (1+\alpha_i x_j)^{-1}
\nonumber \\
&\hspace{10pt}\times\det\Bigg[
\frac{1}{2 \pi i}
\oint_{\widetilde{\gamma}_r}
\frac{
\prod_{k=1}^{i-1}(1-\alpha_k w^{-1}) \prod_{k=1}^{\mu_j^\prime}
(1+\beta_k w^{-1}) \prod_{m=1}^n (1+x_m w)
}{
\prod_{k=1}^{j}(1-\alpha_k w^{-1}) \prod_{k=1}^{\lambda_i^\prime}
(1+\beta_k w^{-1})
w^{\lambda_i^\prime-\mu_j^\prime-i+j+1}} dw
\Bigg]_{i,j=1}^{\ell^\prime}, \label{integralGthree} \\
G_{\lambda\ds\mu}(\xx_{n};\bal,\bbe) & =
\prod_{i=1}^{\ell^\prime} \prod_{j=1}^n (1+\alpha_i x_j)^{-1}
\nonumber \\
&\hspace{10pt}\times\det\Bigg[
\frac{1}{2 \pi i}
\oint_{\widetilde{\gamma}_r}
\frac{
\prod_{k=1}^{i-1}(1-\alpha_k w^{-1}) \prod_{k=1}^{ \mu_j^\prime-1}
(1+\beta_k w^{-1}) \prod_{m=1}^n (1+x_m w)
}{
\prod_{k=1}^{j-1}(1-\alpha_k w^{-1}) \prod_{k=1}^{\lambda_i^\prime}
(1+\beta_k w^{-1})
w^{\lambda_i^\prime-\mu_j^\prime-i+j+1}} dw
\Bigg]_{i,j=1}^{\ell^\prime}, \label{integralGfour}
\end{align}
\end{subequations}
where the contour $\gamma_r$, respectively~$\widetilde{\gamma}_r$, is a circle centered at the origin with radius $r$ satisfying $0 < r < \abs{\alpha_1^{-1}},\abs{\alpha_2^{-1}},\dotsc,\abs{\beta_1^{-1}},\abs{\beta_2^{-1}},\ldots$, respectively $r > \abs{\alpha_1},\abs{\alpha_2},\dotsc,\abs{\beta_1},\abs{\beta_2},\ldots$.
\end{thm}

\section{Free-fermion presentation of the flagged version}
\label{sec:flagged}

In this section, we give a free-fermion presentation of a special case of the flagged version of canonical Grothendieck polynomials.
We take as our definition the Jacobi--Trudi determinant formula is derived in \cite[Thm.~6.7]{HJKSS21}.

\begin{dfn}[{\cite[Thm.~6.7]{HJKSS21}}]
If ${\mathbf r}=(r_1,\dots,r_\ell)$ and ${\mathbf s}=(s_1,\dots,s_\ell)$
are sequences of positive integers satisfying $r_i \leq r_{i+1}$ and $s_i \leq s_{i+1}$ whenever
$\mu_i < \lambda_{i+1}$ for $1 \leq i \leq \ell-1$, then define
\begin{align}
G_{\lambda/\mu}^{\mathrm{row}({\mathbf r},{\bf s})}(\xx;\bal, \bbe) =
\prod_{i=1}^\ell \prod_{k=r_i}^{s_i} (1-\beta_i x_k)
\det \bigl[
h_{\lambda_i-\mu_j-i+j} \bigl( \xx_{[r_j,s_i]} \ds (A_{(\mu_j, \lambda_i]} \sqcup B_{[i,j]}) \bigr)
\bigr]_{i,j=1}^{\ell}.
\label{eq:flaggedonedetrep}
\end{align}
\end{dfn}

We give a free-fermion presentation for $G_{\lambda/\mu}^{\mathrm{row}({\mathbf r},{\bf s})}$ for  the case $r_1=\cdots=r_\ell=1$ and $s_i \leq s_{i+1}$ whenever $\mu_i < \lambda_{i+1}$ for $1 \leq i \leq \ell-1$.
We first give another determinant representation which is an extension of Proposition~\ref{prop:determinant_formula_for_Gs} to flagged version, following the idea of \cite[Section~6.3]{HJKSS21}.

Let $\mathfrak{G}_k^{[s/r]}(\xx,\beta)$ be the \defn{flagged $\beta$-Grothendieck polynomial} that is defined as a coefficient of the following generating function
\[
\mathcal{G}^{[s/r]}(z,\beta) := \sum_{k \in \ZZ} \mathfrak{G}_k^{[s/r]}(\xx,\beta) z^k = \frac{1}{1-\beta z^{-1}}
\prod_{k=r}^s \frac{1 - \beta x_k}{1 - x_k z}.
\]
Note that Remark~\ref{rem:expansions} also applies here as well.

\begin{prop}
If $r_1=\cdots=r_\ell=1$ and
$s_i \leq s_{i+1}$ whenever
$\mu_i < \lambda_{i+1}$ for $1 \leq i \leq \ell-1$, then
\begin{align}
 G_{\lambda/\mu}^{\rm{row}({\bf 1},{\bf s})}(\xx;\bal, \bbe) = \det \left[
\sum_{m=0}^\infty h_m(A_{(\mu_j, \lambda_i]} \sqcup B_{(i,j]}) \mathfrak{G}_{\lambda_i-\mu_j-i+j+m}^{[s_i/1]}(\xx,\beta_i)
\right]_{i,j=1}^{\ell}.
\label{flaggedanotherdetrep}
\end{align}
\end{prop}

\begin{proof}
We show this by transforming~\eqref{flaggedanotherdetrep} to~\eqref{eq:flaggedonedetrep}.
Using
\begin{align}
\mathfrak{G}_m^{[s/r]}(\xx,\beta)=
\prod_{j=r}^s (1-\beta x_j)\sum_{k \ge 0} \beta^k h_{m+k}({\xx}_{[r,s]}), \nonumber
\end{align}
the matrix elements in~\eqref{flaggedanotherdetrep} can be rewritten as
\begin{align}
& \sum_{m=0}^\infty h_m(A_{(\mu_j, \lambda_i]} \sqcup B_{(i,j]}) \mathfrak{G}_{\lambda_i-\mu_j-i+j+m}^{[s_i/1]}(\xx,\beta_i) \nonumber \\
& \hspace{90pt} =
\prod_{k=1}^{s_i} (1-\beta_i x_k) \sum_{m=0}^\infty h_m(A_{(\mu_j, \lambda_i]} \sqcup B_{(i,j]}) \sum_{k=0}^\infty \beta_i^k h_{\lambda_i-\mu_j-i+j+m+k}(\xx_{[1,s_i]}).
\label{transformationtoJTone}
\end{align}
Define $t:=m+k$ and reversing the order of the double sum,
the right hand side of~\eqref{transformationtoJTone} can be rewtitten as
\begin{align}
& \prod_{k=1}^{s_i} (1-\beta_i x_k) \sum_{t=0}^\infty h_{\lambda_i-\mu_j-i+j+t}(\xx_{[1,s_i]}) \sum_{m=0}^t h_m(A_{(\mu_j, \lambda_i]} \sqcup B_{(i,j]}) \beta_i^{t-m} \nonumber \\
& \hspace{60pt} = \prod_{k=1}^{s_i} (1-\beta_i x_k) \sum_{t=0}^\infty h_{\lambda_i-\mu_j-i+j+t}(\xx_{[1,s_i]}) \sum_{m=0}^t h_m(A_{(\mu_j, \lambda_i]} \sqcup B_{(i,j]}) h_{t-m}(\beta_i) \nonumber \\
& \hspace{60pt} = \prod_{k=1}^{s_i} (1-\beta_i x_k) h_{\lambda_i-\mu_j-i+j}\bigl(\xx_{[1,s_i]} \ds (A_{(\mu_j, \lambda_i]} \sqcup B_{[i,j]}) \bigr). \nonumber
\end{align}
Hence
\begin{align}
&\det \left[
\sum_{m=0}^\infty h_m(A_{(\mu_j, \lambda_i]} \sqcup B_{(i,j]}) \mathfrak{G}_{\lambda_i-\mu_j-i+j+m}^{[s_i/1]}(\xx,\beta_i)
\right]_{i,j=1}^{\ell}
\nonumber \\
& \hspace{60pt} = \det \left[
\prod_{k=1}^{s_i} (1-\beta_i x_k) h_{\lambda_i-\mu_j-i+j}\bigl(\xx_{[1,s_i]} \ds (A_{(\mu_j, \lambda_i]} \sqcup B_{[i,j]}) \bigr)
\right]_{i,j=1}^{\ell}
\nonumber \\
& \hspace{60pt} = \prod_{i=1}^\ell \prod_{k=1}^{s_i} (1-\beta_i x_k)
\det \bigl[
h_{\lambda_i-\mu_j-i+j}\bigl(\xx_{[1,s_i]} \ds (A_{(\mu_j, \lambda_i]} \sqcup B_{[i,j]}) \bigr)
\bigr]_{i,j=1}^{\ell} \nonumber
\end{align}
as desired.
\end{proof}

\begin{prop}
\label{prop:flagged_fermions}
If $r_1=\cdots=r_\ell=1$ and
$s_i \leq s_{i+1}$ whenever
$\mu_i < \lambda_{i+1}$ for $1 \leq i \leq \ell-1$, then
\begin{align}
G_{\lambda/\mu}^{\rm{row}({\bf 1},{\bf s})}(\xx;\bal, \bbe)
& =\prod_{i=1}^\ell \prod_{k=1}^{s_i}
(1-\beta_i x_k) \nonumber \\
& \hspace{5pt} \times
\bra{-\ell} Q_\ell \cdots Q_1
(e^{H(\xx_{[1,s_1]})} P_1 e^{-H(\xx_{[1,s_1]})})
\cdots
(e^{H(\xx_{[1,s_\ell]})} P_\ell e^{-H(\xx_{[1,s_\ell]})})
\ket{-\ell},
\label{flaggefermionrepresentation}
\end{align}
where
\begin{align*}
P_i  = e^{H^*(A_{\lambda_i}\sqcup B_{[i,\ell]})} \psi_{\lambda_i} e^{-H^*(A_{\lambda_i}\sqcup B_{[i,\ell]})},
\quad
Q_j = 
e^{H^*(A_{\mu_j}\sqcup B_{(j,\ell]})}
\psi_{\mu_j}^*
e^{-H^*(A_{\mu_j}\sqcup B_{(j,\ell]})}.
\end{align*}
\end{prop}

\begin{proof}
The idea of proof is essentially the same with the one in 
Proposition~\ref{prop:determinant_formula_for_Gs}.
We consider the following generating function
\begin{align}
\Psi
&:=
\prod_{i=1}^\ell \prod_{k=1}^{s_i}
(1-\beta_i x_k)
\times
\bra{-\ell} Q_\ell(w_\ell) \cdots Q_1(w_1) \nonumber \\
& \hspace{20pt} \times (e^{H(\xx_{[1,s_1]})} P_1(z_1) e^{-H(\xx_{[1,s_1]})})
\cdots
(e^{H(\xx_{[1,s_\ell]})} P_\ell(z_\ell) e^{-H(\xx_{[1,s_\ell]})})
\ket{-\ell}, \label{flaggedphi}
\end{align}
where
\begin{align*}
P_i(z_i)  = e^{H^*(A_{\lambda_i}\sqcup B_{[i,\ell]})} \psi(z_i) e^{-H^*(A_{\lambda_i}\sqcup B_{[i,\ell]})},
\quad
Q_j(w_j) = 
e^{H^*(A_{\mu_j}\sqcup B_{(j,\ell]})}
\psi^*(w_j) 
e^{-H^*(A_{\mu_j}\sqcup B_{(j,\ell]})}.
\end{align*}
The
coefficient of $\prod_{i=1}^\ell z_i^{\lambda_i-i} \prod_{j=1}^\ell w_j^{\mu_j-j}$
of $\Phi$ is the right hand side of~\eqref{flaggefermionrepresentation} on one hand.
On the other hand, we show in the following that $\Phi$
is the generating function of the determinant in the right hand side of
\eqref{flaggedanotherdetrep}, which is the determinant representation
of $G_{\lambda / \mu}^{\rm{row}({\bf 1},{\bf s})}(\xx;\bal, \bbe)$.

First, using Wick's Theorem, $\Psi$ defined in~\eqref{flaggedphi} can be rewritten as
\begin{align}
\Psi
& = \prod_{i=1}^\ell \prod_{k=1}^{s_i}
(1-\beta_i x_k)\cdot
\det \Bigl[
\bra{-\ell} Q_j(w_j) e^{H(\xx_{[1,s_i]})} P_i(z_i) e^{-H(\xx_{[1,s_i]})} \ket{-\ell}
\Bigr]_{i,j=1}^\ell. \label{flaggedphitwo}
\end{align}
Next, we substitute 
\begin{align*}
e^{H(\xx_{[1,s_i]})}P_i(z_i)e^{-H(\xx_{[1,s_i]})}
&=
e^{H^\ast(A_{\lambda_i} \sqcup B_{(i,\ell]})}
\left(e^{H^\ast(\beta_i)}e^{H(\xx_{[1,s_i]})}
\psi(z_i)
e^{-H(\xx_{[1,s_i]})}e^{-H^\ast(\beta_i)}\right)
e^{-H^\ast(A_{\lambda_i} \sqcup B_{(i,\ell]})}\\
&=
\left(\sum_{m=0}^\infty h_{m}(A_{\lambda_i}\sqcup B_{(i,\ell]})z_i^{-m}\right)
\prod_{k=1}^{s_i} (1-\beta_i x_k)^{-1} \mathcal{G}^{[s_i/1]}(z_i,\beta_i) \psi(z_i),
\\
Q_j(w_j) & = e^{H^*(A_{\mu_j}\sqcup B_{(j,\ell]})}
\psi^*(w_j) 
e^{-H^*(A_{\mu_j}\sqcup B_{(j,\ell]})}
\\ & = \left(\sum_{k=0}^\infty (-1)^ke_{k}(A_{\mu_j}\sqcup B_{(j,\ell]})w_j^k\right) \psi^*(w_j), 
\end{align*}
and 
\[
\bra{-\ell} \psi^*(w_j) \psi(z_i) \ket{-\ell} = \sum_{p=-\ell}^\infty z_i^p w_j^p, 
\]
into~\eqref{flaggedphitwo} to get
\begin{align}
\Psi
=\det \Bigg[
\left(\sum_{k=0}^\infty (-1)^ke_{k}(A_{\mu_j}\sqcup B_{(j,\ell]})w_j^k\right)
\left(\sum_{m=0}^\infty h_{m}(A_{\lambda_i}\sqcup B_{(i,\ell]})z_i^{-m}\right)
\mathcal{G}^{[s_i/1]}(z_i,\beta_i)
\sum_{p=-\ell}^\infty z_i^p w_j^p
\Bigg]_{i,j=1}^\ell. 
\label{flaggedskewGrothendieckgeneratingdetform}
\end{align}
Extracting the coefficients of $\prod_{i=1}^\ell z_i^{\lambda_i-i} \prod_{j=1}^\ell w_j^{\mu_j-j}$ of both hand sides of~\eqref{flaggedskewGrothendieckgeneratingdetform} gives 
\begin{align}
&\prod_{i=1}^\ell \prod_{k=1}^{s_i}
(1-\beta_i x_k)
\bra{-\ell} Q_\ell \cdots Q_1
(e^{H(\xx_{[1,s_1]})} P_1 e^{-H(\xx_{[1,s_1]})})
\cdots
(e^{H(\xx_{[1,s_\ell]})} P_\ell e^{-H(\xx_{[1,s_\ell]})}) \ket{-\ell}
\nonumber \\
& \hspace{60pt}  = \det \Bigg[ \sum_{m=0}^\infty \sum_{k=0}^{\infty} 
(-1)^k e_k(A_{\mu_j} \sqcup B_{(j,\ell]} )
h_{m-k}(A_{\lambda_i}\sqcup B_{(i,\ell]}   ) 
\mathfrak{G}_{\lambda_i-\mu_j-i+j+m}^{[s_i/1]}(\xx,\beta_i) \Bigg]_{i,j=1}^{\ell} \nonumber \\
& \hspace{60pt} = G_{\lambda/\mu}^{\rm{row}({\bf 1},{\bf s})}(\xx;\bal, \bbe), \nonumber
\end{align}
where the determinants are equal by the basic properties of supersymmetric functions.
\end{proof}

\newcommand{\etalchar}[1]{$^{#1}$}

\end{document}